\documentclass[numbers,webpdf,imaiai]{imanum}
\usepackage{algorithm} 
\usepackage{algorithmic}  
\usepackage[algo2e,ruled,vlined]{algorithm2e}
\usepackage{amsmath,amsthm,amssymb}
\usepackage{subfig,graphicx}
\usepackage{changepage}
\newcommand{\comment}[1]{}

\usepackage{appendix}
\usepackage{xcolor}
\newcommand{\norm}[1]{\left\lVert#1\right\rVert}

\jno{drnxxx}
\begin{document}

\title{Interpolation of Set-Valued Functions}
\shorttitle{Interpolation of Set-Valued Functions}

\author{%
{\sc
Nira Dyn\thanks{Email: niradyn@tauex.tau.ac.il},
David Levin\thanks{Email: levindd@gmail.com}
and Qusay Muzaffar\thanks{Email: qusaym@mail.tau.ac.il}}\\[2pt]
School of Mathematical Sciences\\
Tel Aviv University -  Israel
}
\shortauthorlist{Nira, David and Qusay}

\maketitle

\begin{abstract}
{Given a finite number of samples of a continuous set-valued function F, mapping an interval to compact subsets of the real line, we develop  good approximations of F, which can be computed efficiently.

    In the first stage, we develop an efficient algorithm for computing an interpolant to $F$, inspired by the "metric polynomial interpolation", which is based on the theory in \cite{approximation_of_set_valued_functions}.     
 By this theory, a "metric polynomial interpolant" is a collection of polynomial interpolants to all the “metric chains” of the given samples of $F$. For set-valued functions whose graphs have non-empty interior,  the collection of these "metric chains" can be infinite. Our algorithm computes a small finite subset of "significant metric chains", which is sufficient for approximating $F$.

 For the class of Lipschitz continuous functions with samples at the roots of the Chebyshev polynomials of the first kind, we prove that the error incurred by our computed interpolant  decays with increasing number of interpolation points in the same rate as in the case of interpolation by the metric polynomial interpolant. This is also demonstrated by our numerical examples. 

   For the class of set-valued functions whose graphs have smooth boundaries, we extend our algorithm to achieve a high-precision detection of the points of topology change, followed by a high-order approximation of the boundaries of the graph of F. We further discuss the case of set-valued functions that have "holes" with H\"{o}lder-type singularities. To treat this case we apply some special approximation ideas near the singular points of the "holes", showing by several numerical examples the capability of obtaining high-order approximations.}

{Approximation Theory; Set-Valued Functions; Interpolation;
Algorithm.}
\end{abstract}

\section{Introduction}
\label{sec;introduction}

In \citep{approximations_of_set_valued_functions} the approximation of set-valued functions mapping $[a,b]$ to compact subsets of $\mathbb{R}^d$ is discussed, and theoretical results, regarding the adaptation of the operator of polynomial interpolation from real-valued functions to set-valued functions, have been established. The main idea of this adaptation is to replace operations between numbers by operations between sets. More precisely, given a finite number of samples of a set-valued function $F$, $\{F(x_{i})\}_{i=0}^{N}$, we find all metric chains (see \ref{eq:metric_chains}) connecting these sample sets. The {\bf metric polynomial interpolant} of the set-valued function $F$ at a point $x$ is then defined as the union of the $\mathbb{R}^d$ values at $x$ of the $\mathbb{R}^d$-valued polynomials interpolating the metric chains.

Another way of viewing the problem is reconstruction of a set  in $\mathbb{R}^{d+1}$ from its parallel cross-sections, which are compact sets in $\mathbb{R}^{d}$. For example, 3D objects reconstruction from their 2D cross-sections is an important problem in geometric modelling, where algorithms have been proposed, e.g. \cite{Bajaj}, \cite{Boissonnant}, \cite{KelsDyn}, \cite{Levin1986}.
    
    In this work, we limit our research to set-valued functions mapping $[a,b]$ to compact subsets of $\mathbb{R}$. Our contribution in Section 2 is an efficient algorithm, which finds a small sub-collection of the collection of all metric chains built on the samples $\{F(x_{i})\}_{i=0}^{N}$, which we term \textit{significant metric chains}. These significant metric chains are sufficient for reconstructing an approximation of the graph of the set-valued function.
    
    We demonstrate the results of our new algorithm on Lipshitz continuous set-valued functions, and choose $\{x_i\}_{i=0}^N$ to be the roots of Chebyshev polynomial of degree $N+1$. We show that the algorithm “reconstruct” the graph of $F$, which is a $2D$ object, from its $1D$ samples with an approximation rate of $O\Big(\frac{\log{N}}{N}\Big)$, as predicted by the theory in \cite{approximations_of_set_valued_functions}.

In Section 3 we modify the theoretical and the algorithmic results, to achieve a better approximation rate. In particular, we obtain a rate of $O(h^{4})$, where $h$ is the maximal distance between adjacent interpolation points. This is done under the additional assumption that the smoothness of the boundaries of the graph of $F$ are $C^{4}$. 

According to our conclusions in Section 2, the maximal error occurs in the vicinity of the points of topology change of $F$ (PCTs). Thus, we suggest a method for high order approximation of the points of change of topology, which results in decreasing the interpolation error. Another factor contributing to the improvement in the decay of the error is due to the use of spline interpolation.
We demonstrate our algorithm using a "not-a-knot" cubic spline interpolation at equally spaced points. By modifying the algorithm of the previous section, we can separate the holes in the graph of $F$ from each other, and use individual spline approximations for the boundaries of each hole. 

In Sections 2 and 3 we dealt with set-valued functions (SVFs) with Lipschitz type holes. In Section 4 we extend
our algorithm to deal with SVFs whose holes have upper and lower boundaries, which are $C^{2k}$ with H\"older type singularities at both PCTs. More specifically, assuming the hole is in the interval $[c,d]$, we consider the case where the first derivative of its boundaries diverges as $x\to c^{+}$ at a rate of $|x-c|^{-\frac{1}{2}}$ and as $x\to d^{-}$ at a rate of $|x-d|^{-\frac{1}{2}}$. 
We further assume the hole is defined as the interior of a closed boundary curve $\Gamma\in C^{2k}$, such that every vertical cross-section at $x\in(c,d)$ cuts the curve at two points.

We develop an algorithm for deriving high order approximations to holes of H\"older type singularity. We remark here that the algorithm suggested in \cite{Levin1986} fails in approximating such holes in the neighborhood of the PCT's. The algorithm suggested here starts with deriving a high order approximation to the location of the singular points, i.e., the PCT's.
Next, this information is used for computing local singular approximations of the upper and lower boundary functions $g$ and $h$ near the PCT's. Afterwards, we subtract  these local approximations in order to regularize the given data of $g$ and $h$. Finally, a spline approximation is applied to the regularized data, and the final approximation is obtained by returning the local singular elements. We present a detailed escription of our algorithm and an error analysis for the approximation of the PCTs.

\section{Preliminaries}
In this section we present definitions, notations and operations relevant to our work:
\subsection{Preliminaries on sets and on set-Valued  functions}\label{pre_set_svf}
In this section we present definitions, notation and operations relevant to our work:
\begin{itemize}
\item The set of all compact non-empty subsets of $\mathbb{R}^{d}$ is denoted by  $K(\mathbb{R}^{d})$.
\item For given two sets $V, W\in K(\mathbb{R}^{d})$, the \textbf{Hausdorff metric}, which measures the distance between $V$ and $W$, is defined as
\begin{equation}
	d_{H}(V,W)=\max{\bigg\{\max_{v\in V}{d(v,W)},\max_{w\in W}{d(w,V)}\bigg\}}, 
\end{equation}
where $d(v,W)=\min_{w\in W}{\big\{|v-w|\big\}}$ and $|\cdot|$ is the Euclidean distance.
\item The set of all metric pairs of two given sets $V, W\in K(\mathbb{R}^{d})$ is
\begin{equation}
	\Pi(V,W)=\bigg\{(v,w)\in V\times W: v\in \Pi_{V}(w) \ \text{or}\  w\in \Pi_{W}(v)\bigg\},
\end{equation}
where $\Pi_{V}(w)=\big\{ v\in V:|v-w|=d(w,V) \big\}$.
\item The collection of \textbf{Metric Chains} of a finite sequence of compact sets $\{V_{i}\in K(\mathbb{R}^{d})\}_{i=0}^{N}$ is
\begin{equation}
    \label{eq:metric_chains}
	MC\bigg(\{V_{i}\}_{i=0}^{N}\bigg)=\bigg\{ 
	(v_{0},...,v_{N}) : (v_{i},v_{i+1})\in \Pi(V_{i},V_{i+1}),\  i=0,\dots,N-1\bigg\}.
\end{equation}
Note that $MC\bigg(\{V_{i}\}_{i=0}^{N}\bigg)$ depends on the order of the sets.
\item A \textbf{Metric Linear Combination} of a finite sequence of compact sets $\big\{V_{i}\in K(\mathbb{R}^{d})\big\}_{i=0}^{N}$ is 
\begin{equation}
\label{eq:mlc}
		\bigoplus_{i=0}^{N}\lambda_{i}V_{i}=\bigg\{
		\sum_{i=0}^{N}{\lambda_{i}v_{i}} : (v_{0},...,v_{N})\in MC\bigg(\{V_{i}\}_{i=0}^{N}\bigg)
	\bigg\},
\end{equation} where $\lambda_{i}\in \mathbb{R},\ 0\leq i\leq N$.
\item A set of points $X=\{x_{0},...,x_{N}\}$ is a partition of the interval $[a,b]$ if\\ $a\leq x_{0} < ... < x_{N}\leq b$. The "norm" of $X$ is $|X|=\max_{i}\{{|x_{i+1}-x_{i}|}\}$ for $0\leq i\leq N-1$.
\item A function $F:[a,b]\to K(\mathbb{R}^{d})$ is called a set-valued function (SVF).
\item A set-valued function $F$ is called H\"older continuous, with respect to the Hausdorff metric, if there exists a constant $\mathcal{C}>0$ such that
\begin{equation}
    d_{H}(F(x),F(y))\leq \mathcal{C}|x-y|^{\alpha},\quad x,y\in [a,b].
\end{equation}
where $\alpha\in(0,1]$. We denote the collection of all H\"older continuous functions on $[a,b]$ with the constants $\mathcal{C}$ and $\alpha$ by $Hol_{\alpha}([a,b];\mathcal{C})$. In the special case $\alpha=1$, $F$ is called Lipschitz continuous function. We denote the collection of all Lipschitz continuous functions on $[a,b]$ with a constant $\mathcal{L}>0$ by $Lip([a,b];\mathcal{L})$.
\item For a set-valued function $F:[a,b]\to K(\mathbb{R}^d)$, we define the graph of $F$ by
\begin{equation}
    Graph(F)=\big\{(x,y):x\in[a,b],\ y\in F(x)\big\}.
\end{equation}
\comment{
\item In \cite{the_metric_integral_of_set_valued_functions}, the linear approximation operators that are adapted from real-valued functions to SVFs, are of the form
\begin{equation}
\label{eq:lo}
A_{X}f(x)=\sum_{i=0}^{N}{a_{i}(x)f(x_{i})},
\end{equation}
where the function $f:[a,b]\to\mathbb{R}$, and $X$ is a partition of $[a,b]$.
\item For $F:[a,b]\to K(\mathbb{R}^{d})$ and $X$ a partition of $[a,b]$, the \textbf{Metric Operator} $A^{M}_{X}$ of $F$ has the form
\begin{equation}
A^{M}_{X}F(x)=\bigoplus_{i=0}^{N}a_{i}(x)F(x_{i}).
\end{equation}}
\item $K^{*}(\mathbb{R})$ is a subspace of $K(\mathbb{R})$. Each value of $F:[a,b]\to K^{*}(\mathbb{R})$ is a union of a finite number of compact intervals. Our method presented in this section applies such $F$.
\item A \textbf{Point of Change of Topology} (PCT) of a set-valued function $F:[a,b]\to K^{*}(\mathbb{R})$ is a point $(x,y)\in Graph(F)$ that for small enough $\epsilon > 0$ there exists $\delta > 0$ such that for each $z\in[x-\delta,x+\delta]/\{x\}$, $F(z)$ and $F(x)$ have different topology, i.e. $|F(x)\cap B_{\epsilon}(y)|\neq|F(z)\cap B_{\epsilon}(y)|$ where $B_{\epsilon}(y)=[y-\epsilon,y+\epsilon]$ and $|\cdot|$ represents the number of the intervals in a given set. 
\item A \textbf{Hole} $H$ of a set-valued function $F:[a,b]\to K^{*}(\mathbb{R})$ is a set of the form
    \begin{equation}
    \label{eq:definition_of_hole}
        H=\big\{(x,y):g(x)<y<h(x),x\in(c,d)\big\}\not\subset Graph(F),
    \end{equation}
    where $g,h:[c,d]\to \mathbb{R}$, $g(c)=h(c)$, $g(d)=h(d)$ and $g(x),h(x)\in F(x)$ for $x\in[c,d]$. We note that the points $(c, g(c))$ and $(d, g(d))$ are PCTs of $F$.

\comment{\item We consider $F:[a,b]\to{K}^{*}(\mathbb{R})$  with $M<\infty$ holes $\{H_{j}\}_{j=1}^{M}$ with $\big\{h_{j},g_{j}\big\}_{j=1}^{M}$ their boundary functions, defined on respective intervals $[c_i,d_i]$, $a<c_i$, $d_i<b$. Let $u,\ell:[a,b]\to \mathbb{R}$ be real-valued functions representing the upper and the lower boundaries of $F$. We further assume that $F(a)$ and $F(b)$ are convex.}

\item For $F:[a,b]\to{K}^{*}(\mathbb{R})$, we define $u,\ell:[a,b]\to \mathbb{R}$ be real-valued functions representing the upper and the lower boundaries of $F$. 

    \item {\bf The class $\mathcal{F}([a,b], M)$}: In this work we consider the class of set-valued functions denoted by $\mathcal{F}([a,b], M)$, where $M\in\mathbb{N}$, with the following properties:

\begin{enumerate}
    \item For $F\in \mathcal{F}([a,b], M)$, $Graph(F)$ has separable $M$ holes $\{H_i\}_{i=1}^M$ (i.e. the closures of the holes are disjoint). 
    
    \item A hole $H_i$ is defined on an interval $[c_i,d_i]\subset(a,b)$, with lower and upper boundary functions $g_i$ and $h_i$. Each hole $H_i$ is simple, namely, it is defined as the interior of a closed boundary curve $\Gamma_i$, such that every vertical cross-section at $x\in(c_i,d_i)$ cuts $\Gamma_i$ at two points.
    
    \item The curves $\{\Gamma_i\}$ do not intersect the upper and the lower boundaries of $Graph(F)$.
    
    \item We further assume that $F(a)$ and $F(b)$ are convex.
    
\end{enumerate}

\item We denote the set of functions $\big\{u,\ell\big\}\bigcup\big\{g_{j}\big\}_{j=1}^{M}\bigcup\big\{h_{j}\big\}_{j=1}^{M}$ by $\partial F$. Note that $\partial F$ consists of all the boundary functions of $Graph(F)$.

\end{itemize}

\subsubsection{Representing the SVF by the boundaries functions}\label{subsub1}
\hfill

\medskip
The SVF $F$ may be defined using all the above boundary functions, as follows:
For $x\in [a,b]$ we identify all the holes' intervals $\{[c_i,d_i]\}_{i\in I(x)} $ containing $x$. If $\#I(x)=J(x)>0$, we order the corresponding boundary values $\{g_i(x)\}$, $i\in I(x)$, in ascending order, and index the relevant holes according to this ordering $\{H_{i_j}\}_{j=1}^{J(x)}.$ The set $F(x)$ may be expressed as
\begin{equation}\label{Fatx0}
F(x)=[\ell(x),g_{i_1}(x)]\cup\bigcup_{j=1}^{J(x)-1}[g_{i_j}(x),h_{i_{j+1}}(x)]\cup[h_{i_J}(x),u(x)].
\end{equation}
If $J(x)=0$,
\begin{equation}\label{Fatx1}
F(x)=[\ell(x),u(x)].
\end{equation}

In this work we consider the approximation of a set-valued function from a finite number of its samples, and we consider three cases specified by the smoothness class of the boundary functions $\big\{u,\ell\big\}\bigcup\big\{g_{j}\big\}_{j=1}^{M}\bigcup\big\{h_{j}\big\}_{j=1}^{M}$.
In Section \ref{sec:computed_svf_interpolant} we consider the case of boundary functions of Lipschitz type. In Section \ref{C4boundaries} we assume the boundary functions are $C^4$. In  Section \ref{Holderboundaries} we deal with SVFs whose holes have upper and lower boundaries of H\"older type with H\"older exponent $\frac{1}{2}$ at both PCTs.

\section{Approximated metric polynomial interpolant} \label{sec:computed_svf_interpolant}
Paper \cite{approximations_of_set_valued_functions} presents a theoretical method for interpolation of Set-valued functions by the \textit{metric polynomial interpolant}. Inspired by the definition of the \textit{metric polynomial interpolant}, we present an efficient algorithm for approximating a set-valued function $F$ from a finite number of its samples.We term the output of our algorithm \textit{approximated metric polynomial interpolant}.

For Lipschitz continuous
$F$ and for Chebyshev interpolation points the approximated metric polynomial interpolant approximates $F$ at the same rate as the \textit{metric polynomial interpolant}.

\subsection{The metric polynomial interpolant}
In this section we present the adaptation of the classical polynomial interpolation operators in Lagrange form to set-valued functions, and present an upper bound of the error in an important special case. 
Recalling that for a real valued function $f\in C[a,b]$ the Lagrange form of the polynomial interpolation operator at a partition $X\subset[a,b]$ is given by 
\begin{equation}\label{Lagrange}
    \mathcal{P}_{X}f(x)=\sum_{i=0}^{N}{l_{i}(x)f(x_{i})},
\end{equation}
 For SVF approximation we use the metric analogues of the polynomial interpolation operator:
\begin{definition}(\cite{approximation_of_set_valued_functions}, Section 7.4.3)
Let $F:[a,b]\to K(\mathbb{R}^{d})$ be a set-valued function and $X\subset[a,b]$ be a partition. Let $\big\{(x_{i},F(x_{i}))\big\}_{i=0}^{N}$ be a data set consisting of the samples of $F$ at $X$. The \textbf{metric polynomial interpolation operator} is given by
\begin{equation}\label{PMX}
    \mathcal{P}^{M}_{X}F(x)=\bigoplus_{i=0}^{N}l_{i}(x)F(x_{i})=\bigg\{
    \sum_{i=0}^{N}l_{i}(x)f_{i}:(f_{0},...,f_{N})\in MC\Big(\big\{F(x_{i})\big\}_{i=0}^{N}\Big)
    \bigg\},
\end{equation}
where $l_{i}(x)$ is defined as in the real-valued case.
\end{definition}

It is shown in \cite{approximation_of_set_valued_functions} that for $F\in Lip([a,b],\mathcal{L})$ the metric polynomial interpolation approximates $F$, in the Hausdorff metric, with approximation rate $O(\log(N)/N)$, where $X$ is the set of N Chebyshev points in $[a,b]$ (the roots of the N-th degree 
Chebyshev polynomial in $[a,b]$.
However, the elegant formula (\ref{PMX}) representing the metric polynomial interpolant is not a practical one, since, for most cases, in particular for $F\in \mathcal{F}([a,b],M)$, the set of metric chains is infinite. We present below an efficient algorithm for computing an SVF approximation for $F\in \mathcal{F}([a,b],M)$, which gives the same approximation rate as the metric polynomial interpolant.

\subsection{The algorithm}
In this section we present our new method for an efficient computation of an interpolant to $F$. We do it by approximating the boundaries of the graph of $F$. First, we introduce and recall notions and notation used in the presentation of our method.
\subsubsection{Notions and notation}
\begin{itemize}
\item $\mathcal{N}$ is a \textbf{Maximal Interval} in a set $A\in K^{*}(\mathbb{R})$, if there is no interval $\mathcal{M}\neq \mathcal{N}$ satisfying $\mathcal{N}\subset\mathcal{M}\subset A$.
\item \textbf{Samples of an SVF:}
Given a set of interpolation points $X=\{x_{i}\}_{i=0}^{N}$, the samples of $F\in K^{*}(\mathbb{R})$ at these points are $\{F(x_{i})\}_{i=0}^{N}$. Each sample has the form
\begin{equation}\label{sample}F(x_{i})=\bigcup_{j=0}^{M_i}I_{i,j},
\end{equation}
where, for $j=0,...,M_i$, $I_{i,j}=[a_{2j}^{[i]},a_{2j+1}^{[i]}]$ are maximal intervals in $F(x_i)$, with
$a_{j}^{[i]}< a_{j+1}^{[i]}$, for $j=0,...,M_i$.

\item The set of {\bf approximated points of change of topology} in $F(x_{i})$, for $0\leq i\leq N$, is defined by
$$APCT(F,x_{i})=\bigg\{ p=\frac{\min{(I_{j,k+1})} +\max{(I_{j,k})}}{2},\  0\le k< M_j,\ {\text s.t.}\ p\in F(x_{i})\ {\text and}\ p\notin F(x_{j}),\  {\text for}\  |j-i|=1\bigg\}$$
\item {\bf Extended PCT points:} The sets of right and left extended PCT points in $F(x_{i})$, for $0\leq i\leq N$, are defined by
$$EP_{R}(F,x_{i})=\Big\{p\in F(x_i):\exists j<i\  \text{ s.t. } p\in APCT(F,x_{j}) \text{ and } p\in F(x_\ell),\ j<\ell<i\Big\},$$
$$EP_{L}(F,x_{i})=\Big\{p\in F(x_i):\exists j>i\  \text{ s.t. } p\in APCT(F,x_{j}) \text{ and } p\in F(x_\ell),\ i<\ell<j\Big\}.$$
\item \textbf{Discrete Samples}
Given a sample $F(x_i)$ defined by (\ref{sample}), the discrete sample is the set of end points of the intervals,
$$
\zeta \big(F(x_{i})\big)= \big\{a^{[i]}_{0},a^{[i]}_{1}, \ldots, a^{[i]}_{2M_{i}},a^{[i]}_{2M_{i}+1}\big\}.
$$

Note that $\zeta \big(F(x_{i})\big)$ consists of boundary points of $Graph(F)$.
\item The set of \textbf{Significant Metric Chains} of a given finite set of samples $\big\{F(x_{i})\big\}_{i=0}^{N}$, is a subset of $MC_{F,X}=MC\big(\{F(x_{i})\}_{i=0}^{N}\big)$ given by:
\begin{equation}
\begin{aligned}
    SMC\bigg(\big\{F(x_{i})\big\}_{i=0}^{N}\bigg)=\bigg\{ 
	\big(f_{0},...,f_{N}\big)\in MC_{F,X}:\forall i,\ 0\leq i\leq N-1, (f_{i},f_{i+1})\in \Pi\big(T_{i},T_{i+1}\big)\bigg\},
\end{aligned}
\end{equation}
where $T_{i}=\zeta\big(F(x_{i})\big)\cup APCT(F,x_{i})\cup EP_{r}(F,x_{i})\cup EP_{l}(F,x_{i}),\quad 0\leq i\leq N$.
We denote the set of the polynomials that interpolate the set of significant metric chains by $P_{SMC(\{F(x_{i})\}_{i=0}^{N})}$.
\end{itemize}

\subsubsection{A description of the algorithm for the approximation of $F\in \mathcal{F}([a,b],M)$}

\begin{enumerate}
    \item \textbf{Create the discrete samples}:
    
    We substitute each one of the given samples by a discrete sample. The discrete sample is obtained from the give sample by replacing each interval by its two boundary values.
    \item \textbf{Find all significant metric chains}:
    This step is done by utilizing a tree data structure. Initialy the $i^{\text{th}}$ layer of the tree consists of $2M_{i}+2$ nodes, each containing a point of the $i^{\text{th}}$ discrete sample. At the end of this step each path in the tree represents a significant metric chain. 
    
   First, the algorithm identifies all {\bf APCTs}. Each identified APCT is added as a new node to the corresponding layer, and connected to the two nodes corresponding to its metric pair. Next,  the {\bf extended PCT points} are added to the corresponding layer.
    Then the algorithm scans all the layers of the tree and connects all metric pairs between two consecutive layers.
     
    \item \textbf{Compute the real-valued interpolants of the significant metric chains}:
    
    By using a known algorithm for computing real-valued polynomial interpolation, the algorithm computes a set of polynomial interpolants. Each one of these polynomials interpolates the values of one of the significant metric chains at the points $X$.

    
    \item
    {\bf Extracting the approximations to the boundaries of $F$:}
    Due to the structure of the significant metric chains, 
    the upper and lower boundaries, $u$, $\ell$, of $Graph(F)$ are interpolated each by one of the  interpolation polynomials, computed in the previous step. Also, the boundary of
    each of the M holes in $Graph(F)$ is approximated by at least one pair of interpolants, one approximating its upper boundary and one its lower boundary. Both polynomials interpolate the  two APCTs of that
hole.

The algorithm extracts these approximations to the boundaries of $Graph(F)$ by relating each significant metric chain either to $u$ or to $\ell$ or to an upper or lower boundary of the $M$ holes.
In case there are more than one pair of interpolating polynomials of a boundary of a hole, an arbitrary choice of one pair is taken as the approximant of that boundary. The latter case is nongeneric.
    
    \item{\bf Construct the approximation $\tilde F(x)$:}
    Using the approximations to the boundaries of $Graph(F)$, we apply the procedure in Section \ref{subsub1} for defining a set-valued function using its boundary functions.
    \end{enumerate}
    
    \medskip
    {\bf Conclusion: The interpolation property}
    
    Since the approximated boundary functions interpolate the exact boundary functions at the sample points, it follows that $\tilde F(x_i)=F(x_i)$, $x_i\in X$.

\subsection{Error analysis}
By the theory in \cite{approximation_of_set_valued_functions}, a Lipschitz continuous set-valued function $F$ is approximated with $O\left(\log(N)/N\right)$ approximation rate in the Hausdorff netric,
 by the \textit{metric polynomial interpolant}, 
 interpolating its values at the $N+1$ Chebyshev points (the roots of the Chebyshev polynomial of degree $N+1$). The metric polynomial interpolant is the set of polynomials which interpolate all the metric chains defined by the values of $F$ at the $N+1$ Chebyshev points.
 
Based on the values of $F$ at the $N+1$ Chebyshev points, our algorithm generates an interpolant $\tilde{F}$, interpolating $F$ at the Chebyshev points. Here we show that $\tilde{F}$ approximates $F$ with error (in the Hausdorff metric) of order $O(\log(N)/N)$ as $N\to\infty$, using only a small subset of metric chains (significant metric chains).

Before stating this result as a theorem, we state three results
which are needed in the proof  of the theorem.
The first result is
Theorem 9.3.4 from \cite{approximation_of_set_valued_functions}.   
\begin {theorem}
\label{Th:9.3.4}
If $F\in Lip([a,b];\mathcal{L})$, then for any $f\in\partial F$
\begin{equation}
\label{th:lip_svf}
f\in Lip(D_{f},\mathcal{L}),
\end{equation}
with $D_{f}$ the domain of definition of $f$.
\end{theorem}
\begin{lemma}\label{lemma:Cheb}
Let $f\in Lip([a,b],L)$, and let $X=\{x_i\}_{i=0}^N$ be the $N+1$ Chebyshev points in $[a,b]$. Consider the polynomial $p_N\in\Pi_N$ interpolating a perturbed data $f^*(x_i)=f(x_i)+e_i$ where $|e_i|\le \epsilon$, $0\le i\le N$, then
\begin{equation}
\|f-p_N\|_\infty\le C_1\frac{\log(N)}{N}+C_2\epsilon\log(N).
\end{equation}    
\end{lemma}
\begin{proof}
By the theory of polynomial interpolation to Lipschitz continuous real-valued functions at Chebyshev points \cite{BL}, and since $f\in Lip([a,b],L)$, it follows that if $q_N$ interpolates the values $\{f(x_i)\}_{i=0}^N$, then
$\|f-q_N\|_\infty\le C_1\log(N)/N$ for $N$ large.
The polynomial $p_N$ interpolates values which are $\epsilon$ perturbation of the values of $h$. 
 Considering the Lagrange form of the interpolation operator in (\ref{Lagrange}),
it follows that
\begin{equation}
q_N(x)-p_N(x)=\sum_{i=0}^{N}{l_{i}(x)(q_N(x_{i})-p_N(x_i))},\ \ x\in [a,b].
\end{equation}
By \cite{BL}, the Lebesgue constant of the interpolation operator at Chebyshev points is $O(log(N))$ as $N\to \infty$, namely, $\sum_{i=0}^{N}|l_{i}(x)|\le Clog(N)$.
Therefore, it follows that 
\begin{equation}
|q_N(x)-p_N(x)|\le C_2\epsilon\log(N),\ \ x\in [a,b].
\end{equation}
Finally,
\begin{equation}
|f(x)-p_N(x)|\le |f(x)-q_N(x)|+|q_N(x)-p_N(x)|\le C_1\frac{\log(N)}{N}+C_2\epsilon\log(N).
\end{equation}

\end{proof}

The third result is a general lemma on sets which is proved in Appendix A.
\begin {lemma}
\label{Lemma:sets}
Let $A_1, A_2, B_1,B_2$ be subsets of $\mathbb {R}^d$. Then
$$ d_H\big(A_1\cup A_2,B_1\cup B_2\big)\le \max\big\{d_H(A_1,B_1),d_H(A_2,B_2)\big\}.$$
\end{lemma}

Equipped with these results we turn to the main theorem of this section:

\begin{theorem}\label{Thm1}
Let $F\in Lip([a,b],L)\cap \mathcal{F}([a,b],M)$, and let $\tilde{F}_N$ be the output of our algorithm from input consisting of samples of $F$ at the $N+1$ Chebyshev points. Then
for $x\in[a,b]$,
\begin{equation}
\label{eq:error_a}
    d_{H}(\tilde{F}_N(x),F(x))=O\Big(\frac{\log{N}}{N}\Big) \quad 
\textrm{as}\ \  N\to\infty\ .
\end{equation}

\end{theorem}
\begin{proof}
 To study the approximations of the boundary functions of the holes, let us first consider the case of a single hole, $H$. We assume the hole is spanned in the interval $[c,d]$, and we let $h$ and $g$ be the functions
 defined on $[c,d]$ which describe the upper and the lower boundaries of the hole respectively. The point $(c,h(c))=(c,(g(c))$ and the point $(d,h(d))=(d,g(d))$ are the left and right points of change of topology (PCTs) in the graph of $F$.
 We further simplify the geometry assuming that
 \begin{equation}\label{assumption3}
 h(c),h(d)\in (\max_{x\in [a,b]}\ell(x),\min_{x\in [a,b]}u(x)).
 \end{equation}

By Theorem~\ref{Th:9.3.4} $h,g\in Lip([c,d],\mathcal{L})$ and $u,\ell\in Lip([a,b],\mathcal{L})$. Our algorithm constructs interpolating polynomials at Chebyshev points in $[a,b]$, which requires us to extend $h$ and $g$ to the whole interval. We extend both $h$ and $g$ by the constant $h(c)=g(c)$ on $[a,c]$ and by the constant $h(d)=g(d)$ on $[d,b]$. We denote the extended functions by $h^*$ and $g^*$, and it is easy to verify that these functions are in $Lip([a,b],\mathcal{L})$, since $F\in Lip([a,b],L)$.

Let $X=\{x_i\}_{i=0}^N$ be the $N+1$ Chebyshev points in $[a,b]$, and let $\bar X=\{x_i\}_{i=n}^m=X\cap [c,d]$. Given the set-valued data $\{F(x_i)\}_{i=0}^N$, our algorithm identifies the points $\bar X$, and  extracts the values of the functions $h$ and $g$ at these points, 
$$\{h(x_i),\ g(x_i)\},\ \ \  n\le i \le m.$$ 
The values $h(c)=g(c)$ and $h(d)=g(d)$ are not given, and the algorithm approximates the left PCT, $(c,h(c))$, by the point 
$(x_{n-1},(g(x_{n})+h(x_{n}))/2)$, since $(g(x_{n})+h(x_{n}))/2$ is a metric pair with both $g(x_n)$ and $ h(x_n)$. Similarly the right PCT of the hole, $(d,h(d))$, is approximated by the point 
$(x_{m+1},(g(x_{m})+h(x_{m}))/2) .$ 
The algorithm constructs two significant metric chains on $X$,
$$\{h_\ell\}_{i=0}^{n-1},h(x_{n}),h(x_{n+1})\ldots h(x_m),\{h_r\}_{i=m+1}^N ,$$
and
$$\{g_\ell\}_{i=0}^{n-1},g(x_{n}),g(x_{n+1})\ldots
g(x_m),\{g_r\}_{i=m+1}^N, $$ where
\begin{equation}
\label{eq:hl_and_hr}
   h_\ell=g_\ell=(g(x_{n})+h(x_{n}))/2, \qquad
   h_r=g_r=(g(x_{m})+h(x_{m}))/2.
\end{equation}


Next, we analyze the approximation of $h^*$ by our algorithm. The result  also applies for the approximation of $g^*$.

Our algorithm computes the polynomial $p_N\in \Pi_N$, which interpolates the data $\{(x_i,h_\ell)\}_{i=0}^{n-1}$, $\{(x_i,h(x_i))\}_{i=n}^m$, $\{(x_i,h_r)\}_{i=m+1}^{N}$. In order to prove the theorem, we show that
\begin{equation}\label{OlogNoN}
\|h^*-p_N\|_\infty=O\bigg(\frac{\log{N}}{N}\bigg), \ \ \ \ as\ \  N \to \infty.
\end{equation}

Since  both $g$ and $h$ are in $Lip([c,d],\mathcal{L})$,  
since $h(c)=g(c)$, and since the maximal distance between two adjacent points in $X$ is bounded by $\frac{\pi}{N}$, it follows that
\begin{equation}
|h^*(c)-h_\ell|=|h(c)-\frac{g(x_{n})+h(x_{n})}{2}|=\frac{1}{2}|(g(c)-g(x_{n})+(h(c)-h(x_{n}))|.
\end{equation}
Thus, 
\begin{equation}\label{LpiN}
|h^*(c)-h_\ell|\le \mathcal{L}|c-x_{n}|\le \mathcal{L}|x_{n-1}-x_{n}|\le \mathcal{L}\frac{\pi}{N}.
\end{equation}
Using Lemma \ref{lemma:Cheb} for $f=h^*$ and $\epsilon \le\mathcal{L}\frac{\pi}{N} $ it follows that
\begin{equation}\label{hminusp}
\|h^*-p_N\|_\infty\le C\frac{\log(N)}{N}.
\end{equation}

We denote $\tilde{h}^*_N\equiv p_N$, and analogously, we define $\tilde{g}^*_N$ to be the $N$-th degree polynomial interpolation to $g^*$ at the points $X$. By the arguments leading to
\eqref{hminusp} we conclude that $\|g^*-\tilde g^*_N\|\le C\frac{\log N}{N}$. 

\medskip
Let us consider the restriction of both $g^*$ and $\tilde g_N^*$ to the interval $[\tilde c, \tilde d]\equiv [x_{n-1},x_{m+1}]$, denoting the restrictions as $g_e$ and $\tilde g_{e}$. Note that $g_e$ is an extension of $g$ to $[\tilde c, \tilde d]$ on which the approximated hole is defined. In the same manner we define $h_e$ and $\tilde h_{e}$. The functions $\tilde g_{e}$ and $\tilde h_{e}$ define the lower and the upper boundaries of the approximated hole $\tilde H$, approximating $H$. Moreover, the approximation $\tilde F\sim F$ is defined by the approximated boundary functions using the procedure in Section \ref{subsub1}.

We also use here $\tilde \ell\equiv \tilde{\ell}_N\text{ and }\tilde u\equiv\tilde{u}_N$, the interpolants to $\ell\text{ and }u$, with approximation error of order $O\Big(\frac{\log N}{N}\Big)$ (using \cite{BL}).

For $x\in [a,b]\setminus [\tilde c, \tilde d]$, $F(x)=[\ell(x),u(x)]$ while $\tilde F(x)=[\tilde \ell(x),\tilde u(x)]$,
 and it follows that
$$
d_{H}(\tilde{F}(x),F(x))=O\Big(\frac{\log{N}}{N}\Big),\  
\text{ as } N\to\infty\ .
$$
For $x\in [\tilde c,\tilde d]$
$$
F(x)=\Big[\ell(x),g^*(x)\Big]\cup\Big[h^*(x),u(x)\Big],
$$
whereas
$$
\tilde{F}_N(x)=\Big[\tilde{\ell}(x),\tilde{g}_{e}(x)\Big]\cup\Big[\tilde{h}_{e}(x),\tilde{u}(x)\Big].
$$
Observing that the end-points of $I_1=\big[\ell(x),g^*(x)\big]$ are approximated by the corresponding end-points of $\tilde{I}_1=\big[\tilde{\ell}(x),\tilde{g}_{e}(x)\big]=\big[\tilde{\ell}_N(x),\tilde{g}^*_N(x)\big]$ with error of order $O\Big(\frac{\log N}{N}\Big)$, it is easy to conclude that 

$$
d_H(I_1,\tilde{I}_1)=O\bigg(\frac{\log N}{N}\bigg).
$$

Similarly, for $I_2=\big[h^*(x),u(x)\big]$ and $\tilde{I}_2=\Big[\tilde{h}_{e}(x),\tilde{u}(x)\Big]$, we get $d_H(I_2,\tilde{I}_2)=O\Big(\frac{\log N}{N}\Big)$.
Then, by Lemma~ \ref{Lemma:sets}, we conclude that 
$$d_H\big(F(x),\tilde{F}(x)\big)=O\bigg(\frac{\log{N}}{N}\bigg)\ \ \text{as}\ N\to \infty.$$

Let us extend the error analysis for the case of $M$ holes.
Considering the significant metric chains approximating the holes, 
we define related {\bf auxiliary functions} as follows:

For each hole $H_i$ we have its boundary functions $g_i$ and $h_i$ defined on $[c_i,d_i]$. We extend each of these functions to the left and to the right, using their values at the PCTs, as we did above for the APCTs.
For example, we extend $h_i$ to the right with a constant value $h_i(d_i)$, until it intersects a boundary of another hole $H_j$, or one of the boundary functions $u$ or $\ell$. From the intersection point we follow the boundary function of $H_j$ ($h_j$ or $g_j$) until we reach the right PCT of $H_j$. From this point we continue to the right with the PCT value, and so on till we reach $x=b$.
In case we intersect $u$ or $\ell$, we follow the boundary function up to $x=b$. Similarly, we extend $h_i$ to the left up to $x=a$, and do the same for $g_i$, We denote the extended functions $h_{e,i}$ and $g_{e,i}$, and it is easy to verify that these functions are Lipschitz.
Note that for every significant metric chain there is a corresponding extended function $h_{e,i}$ or $g_{e,i}$.

The polynomial interpolating a significant metric chain data, is also interpolating the associate extended function along the parts lying on the boundary functions. In between boundary functions' segments, the deviation in the interpolated values are due to the deviation between the relevant PCT and the APCT values, which is of order $O(1/N)$ as $N\to \infty$, as proved in the case of one hole. As in the case of one hole, it follows here that the interpolating polynomials approximate the extended functions with approximation order $O(\log(N)/N)$. 

Let us review the definition of the approximation $\tilde F$. 
For each hole $H_i$, there are two significant metric chains, passing through its left and right APCTs, and following the data on its lower and upper boundaries.
Define the polynomials interpolating these data as $\tilde g_{e,i}$ and $\tilde h_{e,i}$ respectively. These polynomials interpolate the extended functions $g_{e,i}$ and $h_{e,i}$ along the parts lying on the boundary functions. As in the case of one hole, it follows here that the interpolating polynomials approximate the extended functions with approximation order $O(\log(N)/N)$.

Denote the left and the right APCTs of $H_i$ by $\tilde c_i$ and $\tilde d_i$, and the restrictions of $\tilde g_{e,i}$ and of $\tilde h_{e,i}$ to $[\tilde c_i,\tilde d_i]$ by $\tilde g_i$ and $\tilde h_i$ respectively. Also denote the restrictions of $g_{e,i}$ and of $h_{e,i}$ to $[\tilde c_i,\tilde d_i]$ by $g^*_i$ and $h^*_i$ respectively.

As in the case of one hole, it follows that as $N\to \infty$,
\begin{equation}\label{Eg}
\|g^*_i-\tilde g_i\|_{\infty,[\tilde c_i,\tilde d_i]}=O(\frac{log(N)}{N}),
\end{equation}
and
\begin{equation}\label{Eh}
\|h^*_i-\tilde h_i\|_{\infty,[\tilde c_i,\tilde d_i]}=O(\frac{log(N)}{N}).
\end{equation}

Viewing the presentation of $F$ by the procedure described in Section \ref{subsub1}, we note that the same SVF $F$ is obtained if we replace there the functions $\{g_i\}$ by $\{g^*_i\}$, $\{h_i\}$ by $\{h^*_i\}$, and the intervals $\{[c_i,d_i]\}$ by $\{[\tilde c_i,\tilde d_i]\}$.

Both $F$ and $\tilde F$ are defined
using the procedure described in Section \ref{subsub1}, replacing all the original boundary functions by their corresponding approximants, and the intervals $\{[c_i,d_i]\}$ by $\{[\tilde c_i,\tilde d_i]\}$. It follows that if $J^*(x)>0$,
\begin{equation}\label{Fatxextended}
F(x)=[\ell(x),g^*_{i_1}(x)]\cup\bigcup_{j=1}^{J^*(x)-1}[g^*_{i_j}(x),h^*_{i_{j+1}}(x)]\cup[h^*_{i_J}(x),u(x)],
\end{equation}
and, if $J^*(x)=0$,
$F(x)=[\ell(x),u(x)]$.

Similarly, the approximation $\tilde F(x)$ is defined using the set of approximated boundary functions: If $J^*(x)>0$,
\begin{equation}\label{Fappr}
\tilde F(x)=[\ell(x),\tilde g_{i_1}(x)]\cup\bigcup_{j=1}^{J^*(x)-1}[\tilde g_{i_j}(x),\tilde h_{i_{j+1}}(x)]\cup[\tilde h_{i_J}(x),u(x)].
\end{equation}
If $J^*(x)=0$, $\tilde F(x)=[\tilde \ell(x),\tilde u(x)]$.
The index $J^*(x)$ in (\ref{Fappr}) is the same as the index in (\ref{Fatxextended}) since we use the same definition intervals $\{[\tilde c_i,\tilde d_i]\}$ for the boundary functions, and the holes are separated.

To complete the proof of Theorem \ref{Thm1}, we observe that each interval in (\ref{Fatxextended}) has a corresponding interval in (\ref{Fappr}). Using the estimates in (\ref{Eg}) and (\ref{Eh}), it follows that the Hausdorff distance between corresponding intervals is $O(log(N)/N)$ as $N\to \infty$, and the proof is completed by using Lemma \ref{Lemma:sets}.
\end{proof}

\subsection{Numerical results}
We demonstrate the interpolation process on one SVF, denoted by $F_A$, and displayed in figures \ref{fig:example_A2}. $F_A$ is explicitly given by,
\[F_A(x)=\begin{cases} 
      [\ell_{A},u_{A}],\quad &x\in[-1,-0.981]\cup[-0.0188,0]\cup[0,0.153]\cup[0.847,1],\\
      \big[\ell_{A},g_{A_{2}}\big]\cup\big[h_{A_{2}},g_{A_{1}}\big]\cup\big[h_{A_{1}},u_{A}\big],\quad &x\in[-0.981,-0.0188],\\
    \big[\ell_{A},g_{A_{3}}\big]\cup\big[h_{A_{3}},\ell_{A}\big],\quad &x\in[0.153,0.847].
   \end{cases}
\]
where, 
\begin{align*}
&u_{A}=\tanh{(-x)} + 1,\quad\ell_{A}=-\tanh{(-x)}-1\\
&h_{A_{1}}=-\frac{1}{\cosh(2x + 1)},\quad g_{A_{1}}=\frac{1}{\cosh(2x + 1)}-\frac{4}{3}\\
&h_{A_{2}}=-\frac{1}{\cosh(2x + 1)}+\frac{4}{3},\quad g_{A_{2}}=\frac{1}{\cosh(2x + 1)}\\
&h_{A_{3}}=-\frac{1}{\cosh(2x - 1)}+ \frac{4}{5},\quad g_{A_{3}}=\frac{1}{\cosh(2x - 1)} - \frac{4}{5}
\end{align*}
\subsubsection{The figures}
\begin{figure}[!ht]
  \centering
  \subfloat[][The Set-Valued Function $F_A$]{\includegraphics[width=.4\textwidth]{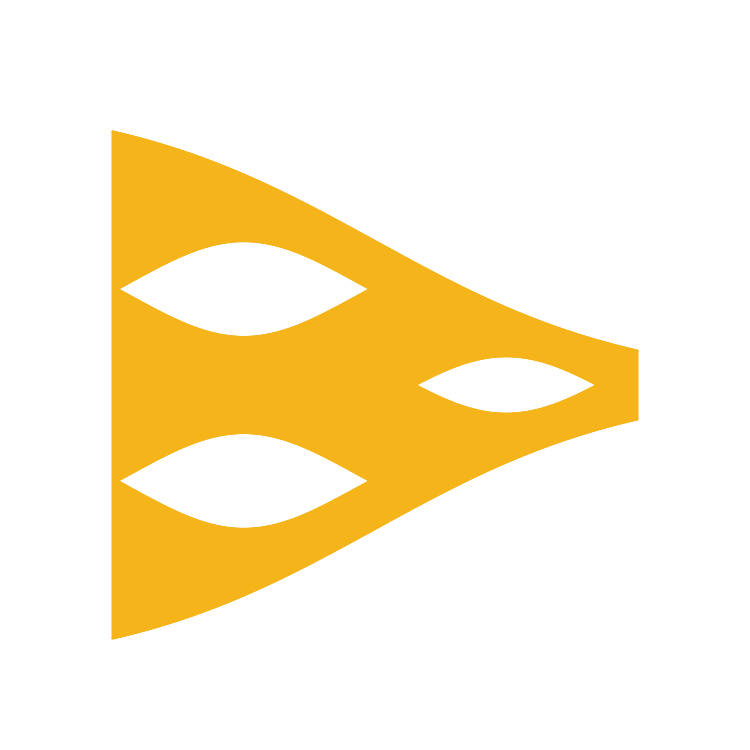}}\quad
  \subfloat[][Approximation with 10 samples]{\includegraphics[width=.4\textwidth]{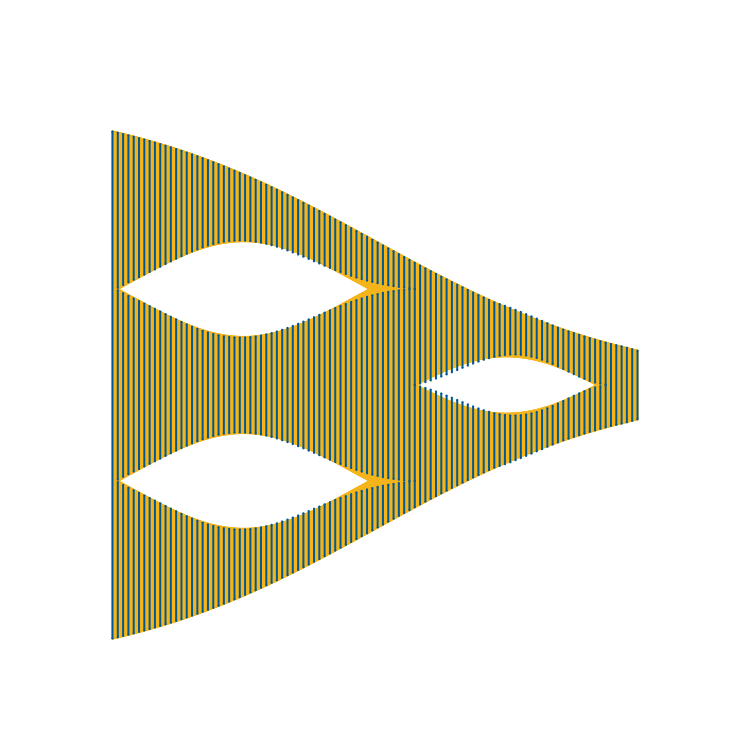}}\\
  \subfloat[][Approximation with 20 samples]{\includegraphics[width=.4\textwidth]{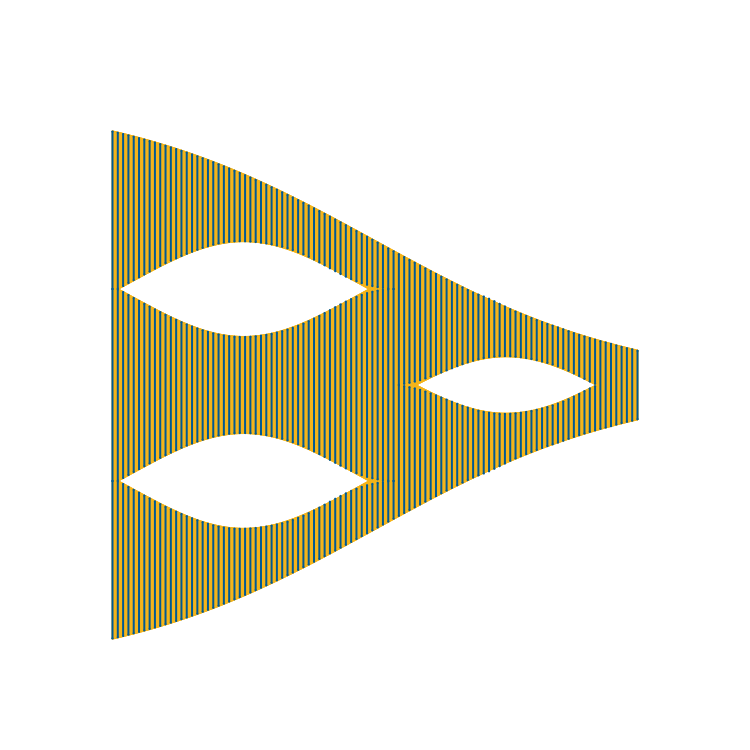}}\quad
  \subfloat[][Approximation with 30 samples]{\includegraphics[width=.4\textwidth]{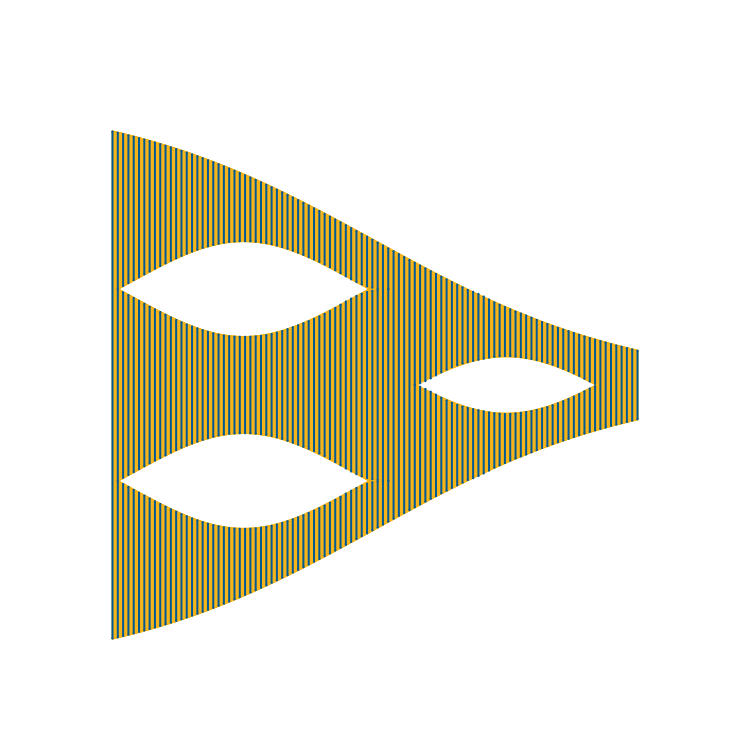}}
  \caption{The set-valued function $A$ and its approximations. Each approximation is represented by vertical blue lines, drawn on the graph of the original function, which is colored in yellow.}
  \label{fig:example_A2}
\end{figure}
\begin{figure}[!ht]
    \centering
    \includegraphics[width=0.8\textwidth]{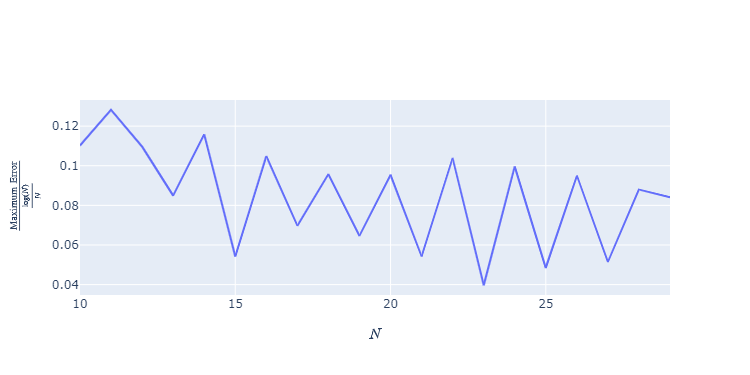}
    \caption{The ratio between the interpolation error and $\frac{\log{(N)}}{N}$ as a function of the number of\\ interpolation points $N$ for the set-valued function $A$.}
    \label{fig:error_example_A2}
\end{figure}

The figures below consist of four sub-figures. The last three of the sub-figures show three interpolants corresponding to different number of interpolation points $N=n+1$. Each interpolant is represented by vertical blue lines, drawn on the graph of the original function, which is colored in yellow. The first sub-figures shows the graph of $F_A$. We show the interpolation error, termed
$$
\text{Maximum Error}=\max_{j}{\Big\{d_{H}\big(F(\xi_{j}),\Tilde{F}(\xi_{j})\big)\Big\}},
$$
where $\{\xi_{j}\}_{j=1}^{2N}$ is a set of $2N$ equidistant points in $[a,b]$. We plot 
$$
G=\frac{\text{Maximum Error}}{\frac{\log{N}}{N}},
$$
as a function of the number of the interpolation points $N$.
\subsubsection{Conclusions from the figures}
\label{conclusions_section_1}
\begin{enumerate}
    \item Figure \ref{fig:example_A2} demonstrates that the interpolation error decreases as $N$ increases in accordance with the theory.
    
    Most importantly, we observe that the maximal error occurs near the PCTs. This observation suggests that to reduce the maximal error we need to approximate the location of the PCTs more accurately. This is done in the next section.
    
    \item As seen from figure
    \ref{fig:error_example_A2}, the maximal error divided by $\frac{\log{N}}{N}$ is bounded for each $N$. This indicates that the error decays at the rate as predicted by the theoretical result (\ref{eq:error_a}). 
\end{enumerate}

\section{Improved Approximation of set-valued functions with $C^4$ boundaries}\label{C4boundaries}

In this section we modify the algorithm and improve the theoretical results of the previous chapter, to achieve a better rate of decay of the error in the interpolation. We obtain a rate of $O(|X|^{4})$, where $X$ is the partition determining the interpolation points. This is done under the assumption that the smoothness of the boundaries of $Graph(F)$ is $C^{4}$. 

According to our conclusions in \ref{conclusions_section_1}, the maximal error occurs in the vicinity of the PCTs. Thus, we suggest a method for achieving high order approximation of the points of change of topology, which results in decreasing the overall interpolation error. Another factor contributing to the improvement in the approximation order is due to the use of spline interpolation,
using a "not-a-knot" cubic spline interpolation.

By modifying the algorithm of the previous section, we can separate the holes from each other. Therefore, we firstly discussed the case of only one hole $H$, defined by (\ref{eq:definition_of_hole}), with $C^4$ boundaries. Later on we extend the result to approximating a set-valued function $F$ whose graph has several holes with $C^4$ boundaries.

\subsection{Notion and notations}\label{Notions}
\begin{enumerate}
    \comment{\item We say that $\partial F$ is $H\ddot{o}l_{\alpha}$ if each $f\in\partial F$ is $H\ddot{o}l_{\alpha}(I)$ where $I\subset[a,b]$ is the domain of definition of $f$.}
    \item For the simplicity of the presentation, we use here the uniform partition,
    \begin{equation}
    X=\{x_i=a+i\Delta\}_{i=0}^N,\ \ \Delta=\frac{b-a}{N},
    \end{equation}
    although the algorithm and the approximation results apply for a general partition $X$.
    \item A set of \textbf{Boundary Metric Chains} of a given finite set of samples $\big\{F(x_{i})\big\}_{i=0}^{N}$, at a partition $X$, is given by:
    \begin{equation}
    \begin{aligned}
    BMC\bigg(\big\{F(x_{i})\big\}_{i=0}^{N}\bigg)=\bigg\{
	\big(f_{n},...,f_{m}\big): \exists\big(f_{0},...,f_{n},...,f_{m},...,f_{N}\big)\in SMC\Big(\big\{F(x_{i})\big\}_{i=0}^{N}\Big) \\ \quad\land\quad \exists \eta\in\partial F \text{ s.t } f_{k}=\eta(x_{k}),\quad k=n,...,m  \bigg\}.
    \end{aligned}
    \end{equation}
    \item Let the hole be defined on the interval $(c,d)$ where
    $a<c<d<b$, and let $\{x_i\}_{i=n}^m$ be all points of $X$ in $(c,d)$.
    \item It is easy to see that a boundary metric chain derived for the hole $H$ is a significant metric chain restricted to $[c,d]$.

\end{enumerate}
\subsection{A description of the algorithm}\label{Thealgorithm}

\begin{enumerate}
\item \textbf{Identifying the hole}:

We modify the algorithm of the previous section so it identifies the cross-sections cutting the hole $H$. Then, instead of producing the set of \textbf{Significant Metric Chains}, the algorithm produces the set of \textbf{Boundary Metric  Chains}. Thus, we get the two boundary metric chains of the hole $H$, $\big\{g(x_n),g(x_{n+1}),\cdots,g(x_m)\big\}$ and $\big\{h(x_n),h(x_{n+1}),\cdots,h(x_m)\big\}$ for the lower and the upper boundaries respectively.
\item \textbf{Approximating the right and left PCTs of $H$}:

In order to find an approximation for the left PCT, we interpolate the first four values of $g$ and of $h$ at $\{x_i\}_{i=n}^{n+3}$, using two cubic polynomial interpolants, $\Tilde{g}_{L}$ and $\Tilde{h}_{L}$, approximating the lower and upper left parts of the hole boundaries. Then,
we find the intersection point of these two polynomials. If such an intersection exists and is in the interval $[x_{n}-\Delta,x_{n}]$ then it gives a better approximation of the location of the left PCT, $\big(\Tilde{c},\Tilde{g}_{L}(\Tilde{c})\big)$, of $H$. However, if $\Tilde{c} < x_{n}-\Delta$ then we we take the approximation of the PCT to be $x_n-\Delta$, which is computed by the algorithm of the previous section. See Lemma \ref{lem:location_of_pct_inverval}. A similar procedure is applied near the right PCT of $H$, using cubic interpolation at the points $\{x_i\}_{i=m-3}^m$, to yield the approximation $\big(\Tilde{d},\Tilde{g}_{R}(\Tilde{d})\big)$ of the right PCT.

\item \textbf{Approximating the lower and the upper boundaries of $H$}:

 In order to approximate the boundaries of $H$, we use "not-a-knot" cubic spline interpolation on the extended data-sets $\big\{\big(\Tilde{c},\Tilde{g}_{L}(\Tilde{c})\big),\big(\Tilde{d},\Tilde{g}_{R}(\Tilde{d})\big)\big\}\cup\big\{(x_i,g(x_i))\big\}_{i=n}^{m}$ and $\big\{\big(\Tilde{c},\Tilde{h}_{L}(\Tilde{c})\big),\big(\Tilde{d},\Tilde{h}_{R}(\Tilde{d})\big)\big\}\cup\big\{(x_i,h(x_i))\big\}_{i=n}^{m}$ to obtain the approximation of the lower and the upper boundaries respectively.

\end{enumerate}
\subsection{Error Analysis for one hole}
In this sub-section, we find an error estimate of the output of the above algorithm. First we analyse the error in approximating the location of the PCT's,
and then we consider the approximation of the boundaries of the hole.

\subsubsection{Error estimate of the PCT approximation}

We find the approximation order of the left PCT $\big(c,g(c)\big)$ (the proof is similar for the right PCT). We estimate the Euclidean distance between the original and approximated PCT, 
\begin{equation}
    \label{eq:B_pct_error_expression}
    E=\sqrt{|c-\Tilde{c}|^{2}+|g(c)-\Tilde{g}_{L}(\Tilde{c})|^{2}}.
\end{equation}
\begin{lemma}
\label{lem:location_of_pct_inverval}
The left PCT is in the interval $[x_n-\Delta,x_n)$ and the right PCT is in the interval $(x_m,x_m+\Delta]$
\end{lemma}

The proof follows from the definition of $\{x_i\}_{i=n}^m$.

According to the above lemma, we focus on the interval $[x_{n}-\Delta,x_{n}]$ for approximating the left PCT. We denote by $I^{*}$ the open interval between $c$ and $\Tilde{c}$, and let $\phi=h-g$ and $\psi=\Tilde{h}_{L}-\Tilde{g}_{L}$.

\begin{proposition} 
\label{prop:pct_estimate_b2}
If $\phi^{'}(c)=C>0$, then, step 2 in the above algorithm approximates the left PCT with error $E=O(\Delta^{4})$ as $\Delta\to 0$.
\end{proposition}
\begin{proof}
Finding the intersection of $\Tilde{g}_{L}$ and $\Tilde{h}_{L}$ is equivalent to finding the root of $\psi$. Note that $\psi$ is a cubic polynomial which interpolates $\phi$ at $\{x_n,x_{n+1},x_{n+2},x_{n+3}\}$, and by definition, $\psi(\Tilde{c})=0$ and $\phi(c)=0$. $\phi\in C^4[a,b]$, and it can be extended as a $C^4$ function on $\mathbb{R}$. Using the error formulae for polynomial interpolation it follows that
\begin{equation}\label{OD4}
|\phi(x)-\psi(x)|=O(\Delta^4)
\end{equation}
as $\Delta\to 0$,
in an $O(\Delta)$ neighborhood of $x_n$.

By the assumption that $\phi^{'}(c)=C>0$, it follows that, for a small enough $\Delta$, $\phi^{'}(x)>\tilde C>0$ in an $O(\Delta)$ neighborhood of $x_n$.
Observing that 
$$\phi(c\pm \Delta)=\pm C\Delta +O(\Delta^2),$$
and in view of (\ref{OD4}), it follows that for a small enough $\Delta$, $\psi$ changes sign in $[c-\Delta,c+\Delta]$, hence
$\tilde c\in [c-\Delta,c+\Delta]$.

To estimate $|c-\tilde c|$ we employ the Mean Value Theorem,
\begin{equation}\label{MVT}
    \frac{|\phi(c)-\phi(\tilde c)|}{|c-\Tilde{c}|}=\phi'(\xi)>\tilde C.
\end{equation}
Since $\phi(c)=0$, it follows that 
$|c-\tilde c|<\frac{|\phi(\tilde c)|}{\tilde C}.$
Using (\ref{OD4}), $
|c-\tilde c|<\frac{|\psi(\tilde c)+O(\Delta^4)|}{\tilde C},$ and recalling that $\psi(\tilde c)=0$,
\begin{equation}\label{cminusct}
|c-\tilde c|=O(\Delta^4).
\end{equation}

To evaluate the error in $y$-coordinate of the approximated PCT, we estimate $|g(c)-\Tilde{g}_{L}(\Tilde{c})|$.
Using the Mean Value Theorem, and since $\tilde g'$ is bounded near $c$, we obtain
\begin{equation}
\label{eq:pct_y_error_b}
    |g(c)-\Tilde{g}_{L}(\Tilde{c})|=|\Tilde{g}_{L}(c)-\Tilde{g}_{L}(\Tilde{c})+O(\Delta^4)|\leq |\Tilde{g}_{L}'(\xi)||c-\Tilde{c}|+O(\Delta^4)=O(\Delta^4),
\end{equation}
where $\xi$ is between $c$ and $\Tilde{c}$.
Combining the last two error estimates, we conclude for the error $E$ in approximating the location of the left PCT that $$E=O(\Delta^4),$$
as $\Delta\to 0$.
\end{proof}

\subsubsection{Estimating the approximation error}\hfill

\medskip
As defined above, we approximate the lower and upper boundaries of $H$, using "not-a-knot" cubic spline interpolation on the data-sets $$\big\{\big(\Tilde{c},\Tilde{g}_{L}(\Tilde{c})\big),\big(\Tilde{d},\Tilde{g}_{R}(\Tilde{d})\big)\big\}\cup\big\{(x_i,g(x_i))\big\}_{i=n}^{m}$$
and 
$$\big\{\big(\Tilde{c},\Tilde{h}_{L}(\Tilde{c})\big),\big(\Tilde{d},\Tilde{h}_{R}(\Tilde{d})\big)\big\}\cup\big\{(x_i,h(x_i))\big\}_{i=n}^{m}.$$ We denote the resulting approximations to $g$ and $h$ on the interval $[\tilde c, \tilde d]$ by $\tilde g$ and $\tilde h$ respectively. This is performed for each hole in $Graph(F)$. We also use "not-a-knot" cubic spline interpolation to approximate the functions $\ell$ and $u$ describing the lower and upper boundaries of $Graph(F)$, denoting the appropriate approximations by $\tilde \ell$ and $\tilde u$. The approximation of the boundaries of $Graph(F)$ induces the definition of the approximation $\tilde F$ of the set-valued function $F$. 

For simplicity of presentation  we 
introduce the definition of $
\tilde F$ and the error analysis for the case of one hole. A full error analysis for the case of several holes is presented in Section \ref{Mholes} for the case of holes of H\"older type singularities. The method of approximating the PCTs and the boundaries of a hole with H\"older type singularities is different, but the method of extending the approximation results to the case of several holes holds for the case of $C^4$ boundaries as well.

\begin{definition}\label{Def2}{\bf The approximation $\tilde F(x)$.}
\begin{equation}\label{Def3eq}
\tilde F(x)=\begin{cases} 
      [\tilde \ell(x),\tilde u(x)], &x\in[a,\tilde c]\cup[\tilde d,b],\\
      \big[\tilde \ell(x),\tilde h(x)\big]\cup\big[\tilde g(x),\tilde u(x)\big], &x\in[\tilde c, \tilde d].
   \end{cases}
\end{equation}
\end{definition}

We recall that for $f\in C^4([\alpha,\beta])$ the "not-a-knot" cubic spline interpolation $s(f;x)$ satisfies the following error estimate (see \cite{numerical_analysis}, Chapter 2.3.4):
    \begin{equation}\label{notaknot}
        \norm{f(x)-s(f;x)}_{\infty,[\alpha,\beta]}\leq C\Delta^{4} \norm{f^{(4)}}_{\infty,[\alpha,\beta]},  
    \end{equation}
    where $\Delta$ is the length of the maximal knots' interval, and $C$ is a constant independent of $f$ and $\Delta$.\\
\begin{proposition}
\label{prop:error_estimate_b}
 For $x\in [a,b]$, $d_H(F(x),\Tilde{F}(x))=O(\Delta^{4})$ as $\Delta\to 0$.
\end{proposition}
\begin{proof}
Case 1: If $c < \tilde c$ and $x\in [c,\tilde c]$, then $\tilde F(x)$ has a 1-dimensional hole which is of length $|\phi(x)|$. Since $\phi(c)=0$ and $|\phi'(c)|$ is bounded, it follows from (\ref{cminusct}) that $\phi(x)=O(\Delta^4)$. By (\ref{notaknot}) we have that the approximations to $\ell$ and to $u$ are also $O(\Delta^4)$. Altogether,  
$$d_H(F(x),\Tilde{F}(x))=\max\{|\ell(x)-\tilde \ell(x)|,|u(x)-\tilde u(x)|, 0.5|\phi(x)|\}=O(\Delta^4).$$

Case 2: If $\tilde c < c$ and $x\in [\tilde c,c]$, then $\tilde F(x)$ has a 1-dimensional hole which is of length $|\psi(x)|$. Since $\psi(\tilde c)=0$ and $|\phi'(c)|$ is bounded, it follows as above that  
$$d_H(F(x),\Tilde{F}(x))=\max\{|\ell(x)-\tilde \ell(x)|,|u(x)-\tilde u(x)|, 0.5|\psi(x)|\}=O(\Delta^4).$$

Case 3: If $x<\min\{c.\tilde c\}$, or $\max\{c,\tilde c\}<x<\min\{d,\tilde d\}$, or $x>\max\{d,\tilde d\}$
$$d_H(F(x),\Tilde{F}(x))=\max\{|\ell(x)-\tilde \ell(x)|,|u(x)-\tilde u(x)|\}=O(\Delta^4).$$
The approximation near the right PCT of the hole is treated in the same manner as Cases 1 and 2.
\end{proof}
\comment{
\begin{figure}[ht]%
    \centering
    \includegraphics[width=\textwidth]{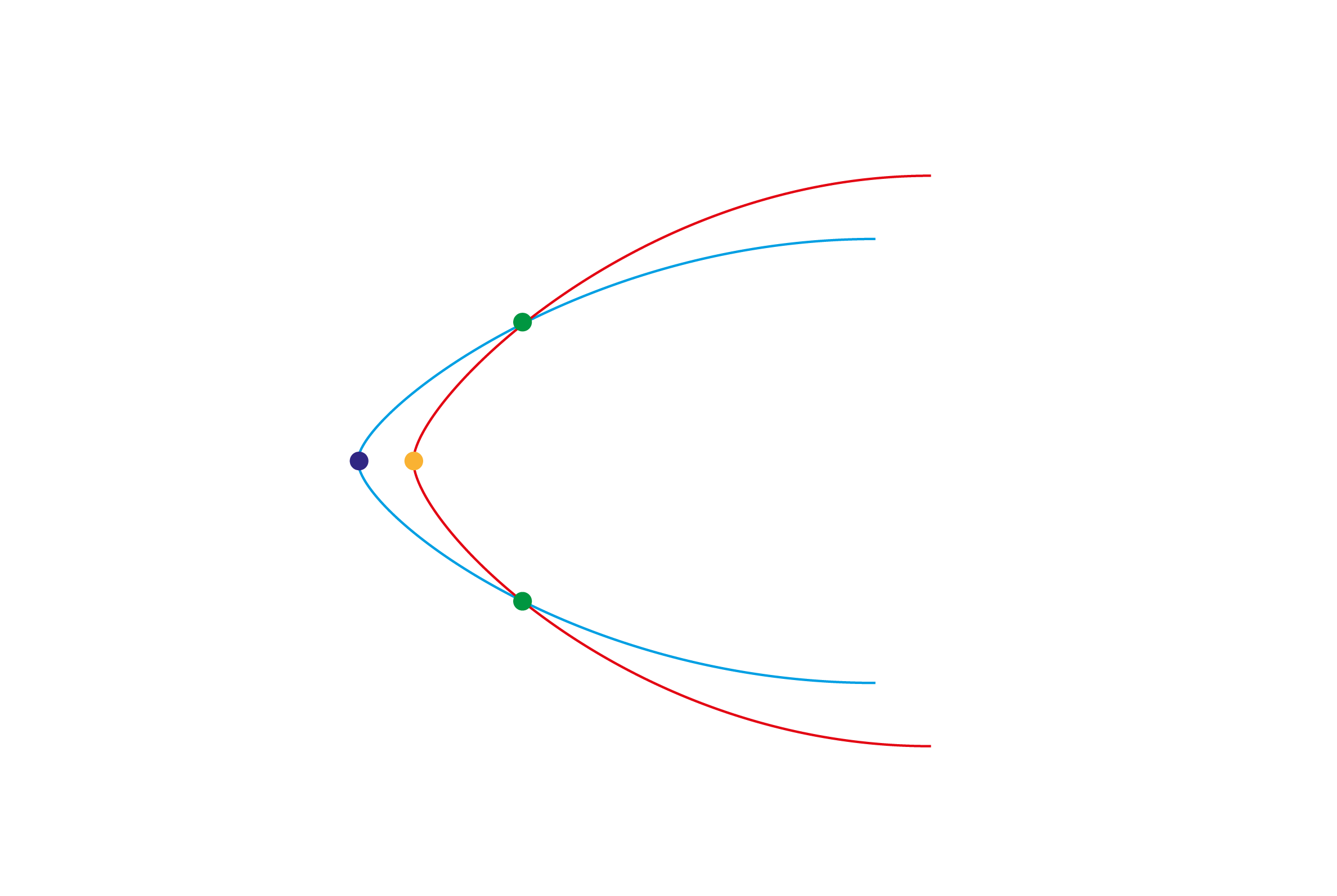}
    \qquad
    \caption{An illustration of the error of SVF interpolation near a right PCT. The blue curve can represent the boundary of the hole of $F$ and the red one that of $\Tilde{F}$ or vice versa. Note that the maximal error occurs at the yellow point, which is either the approximated PCT or the actual PCT respectively.}
    \label{fig:error_pct}
\end{figure}
}
\subsection{Numerical results}
We demonstrate the process of approximating the left PCT as well as showing the decay rate of the interpolation error on one SVF $F$ displayed in figure \ref{fig:example_B1}, which is given explicitly by,

\[ F(x)=\begin{cases} 
      \big[-e^{x},e^{x}],\qquad &x\in[-1,1]/[-x_a,x_a],\\
       \Big[-e^{x},-\frac{\cos{(3x)}}{3}\Big]\bigcup\Big[\frac{\cos{(2x)}}{2}, e^{x}\Big],\qquad &x\in[-x_a,x_a],
   \end{cases}
\]
where $-x_a$ and $x_a$ are the roots of $f(x)=\frac{\cos{(2x)}}{2}+\frac{\cos{(3x)}}{3}$.
\subsubsection{The figures}
\begin{figure}[!ht]
  \centering
  \subfloat[][The Set-Valued Function $F_A$]{\includegraphics[width=.4\textwidth]{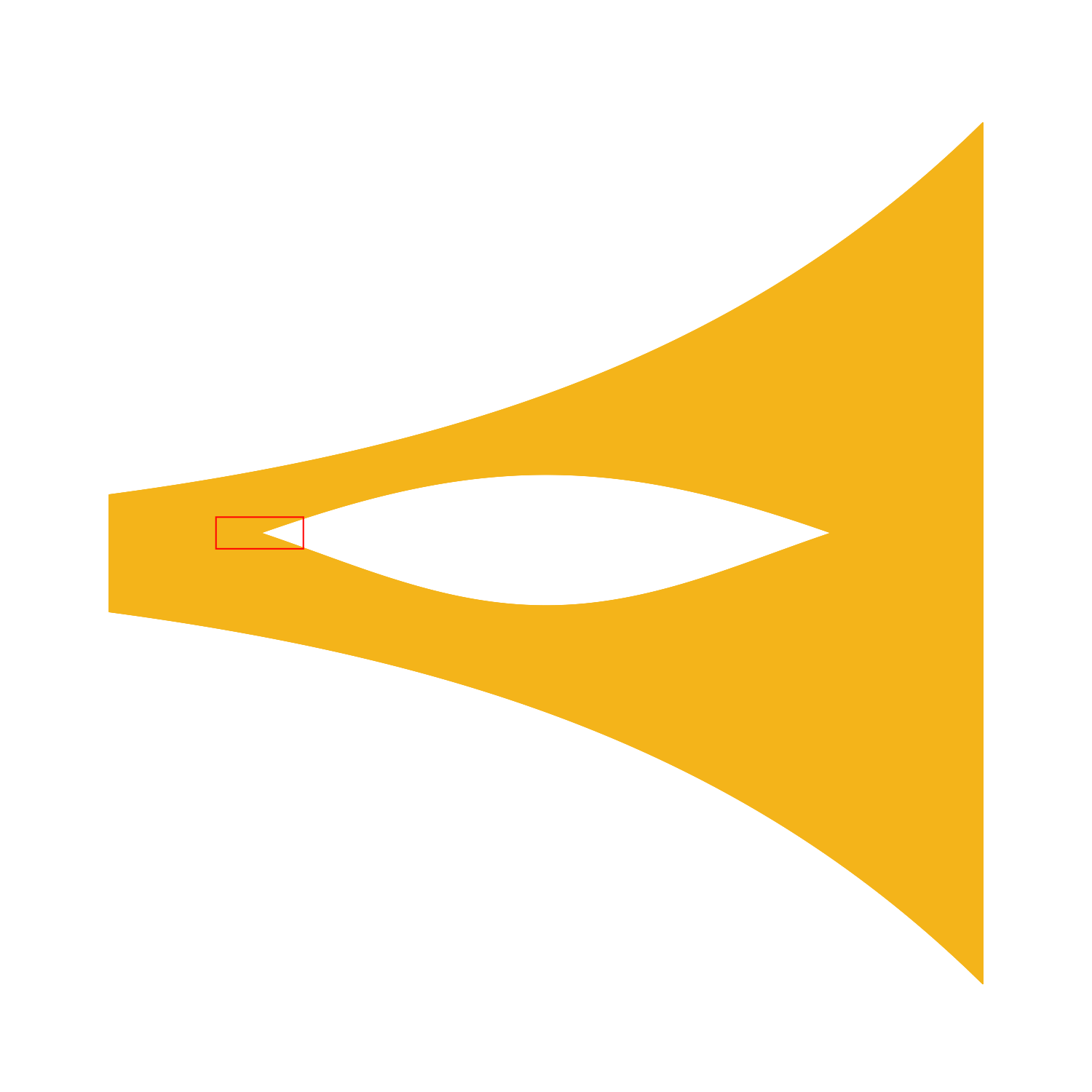}}\quad
  \subfloat[][Approximation with 10 samples]{\includegraphics[width=.4\textwidth]{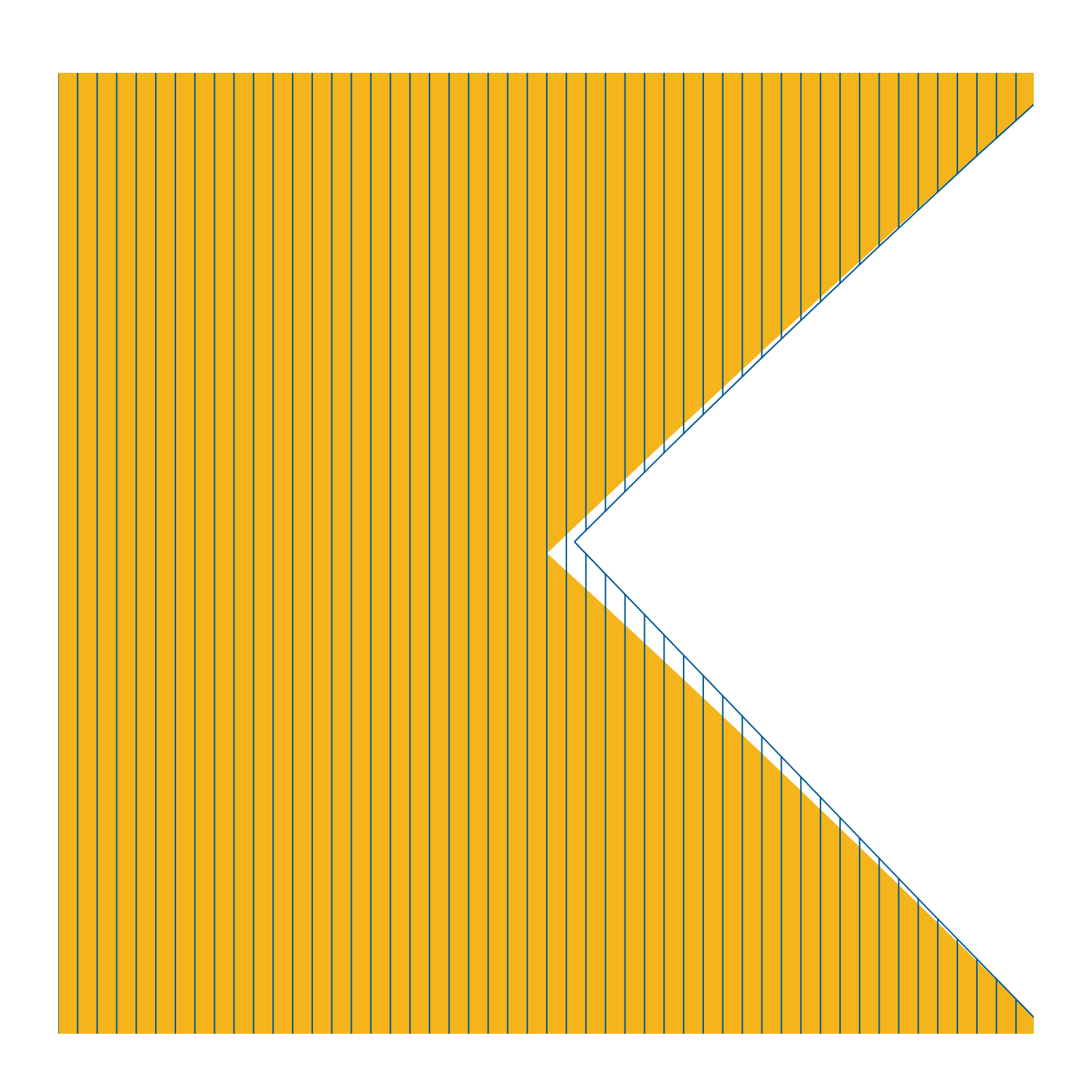}}\\
  \subfloat[][Approximation with 14 samples]{\includegraphics[width=.4\textwidth]{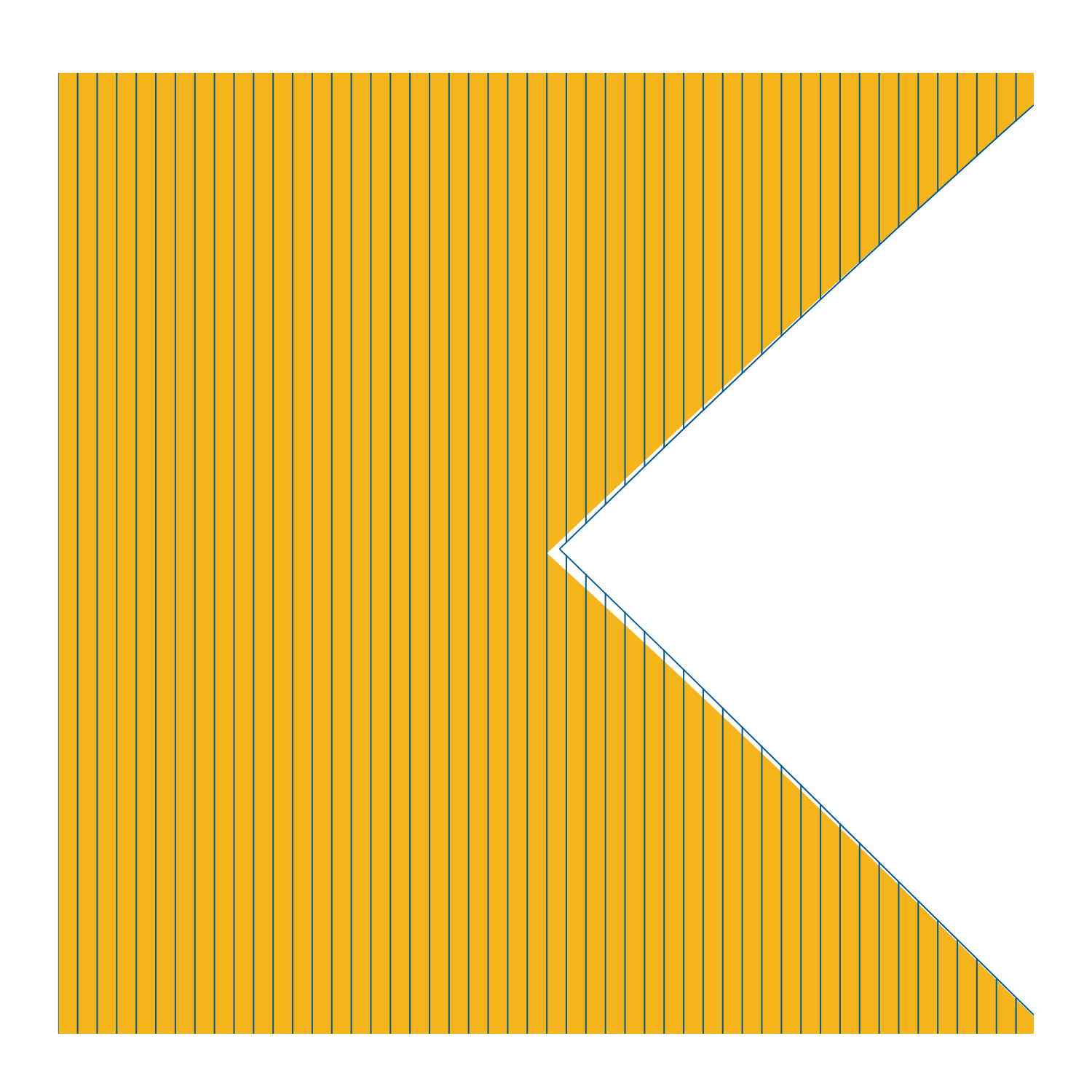}}\quad
  \subfloat[][Approximation with 18 samples]{\includegraphics[width=.4\textwidth]{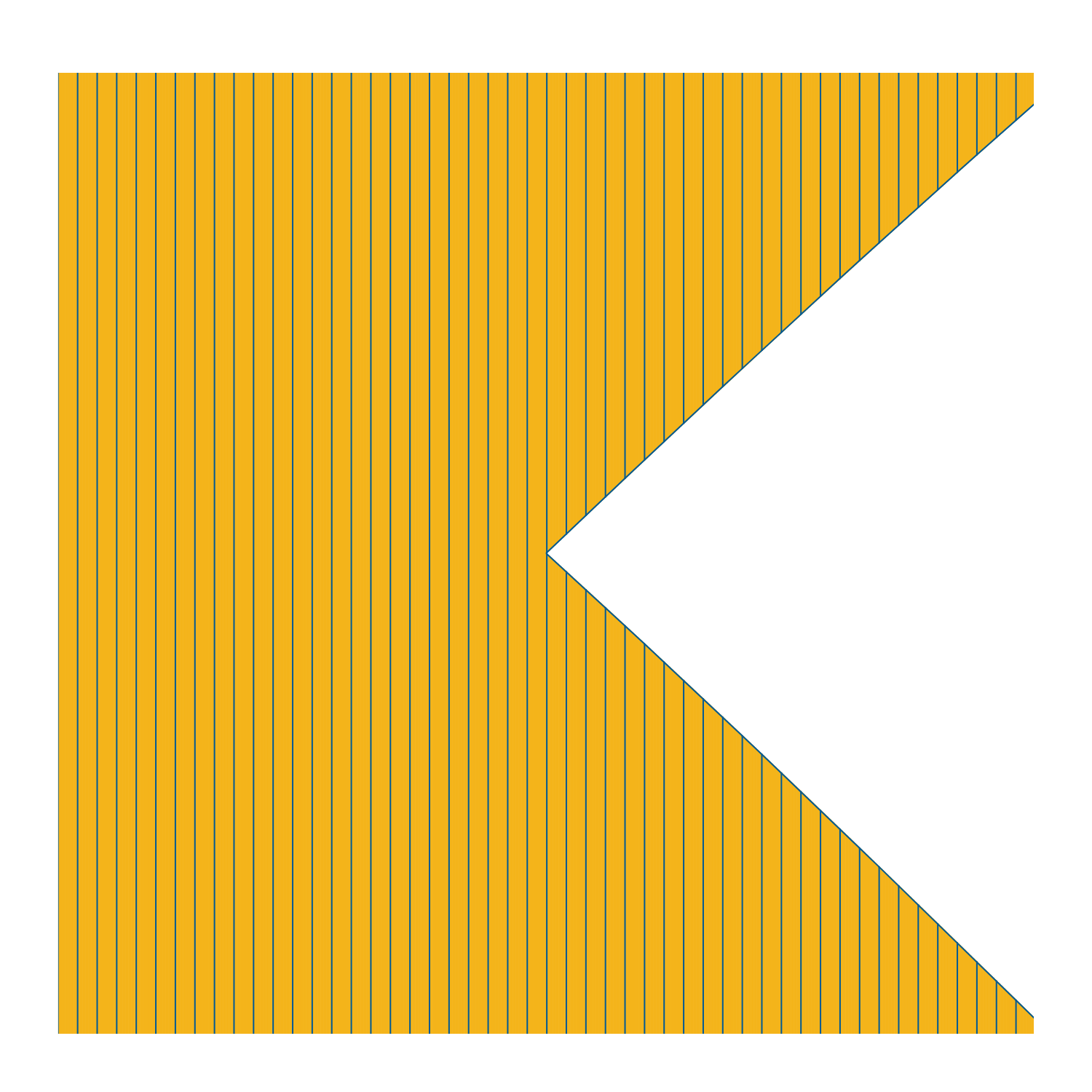}}
  \caption{The set-valued function $B$ and its approximations, zoomed in on the red rectangle in (a). Each approximation is represented by vertical blue lines, drawn on the graph of the  original function, which is colored in yellow.}
  \label{fig:example_B1}
\end{figure}
\begin{figure}[!ht]
    \centering
    \includegraphics[width=0.8\textwidth]{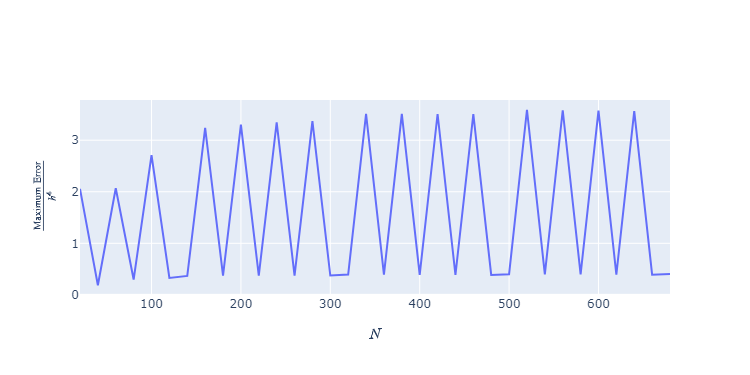}
    \caption{The interpolation error divided by $\Delta^{4}$, as a function of the number of interpolation points ($N$) for the set-valued function $F$. Here, $\Delta=\frac{b-a}{N-1}$.}
    \label{fig:maximum_error_example_B1}
\end{figure}
\begin{figure}[!ht]
    \centering
    \includegraphics[width=0.8\textwidth]{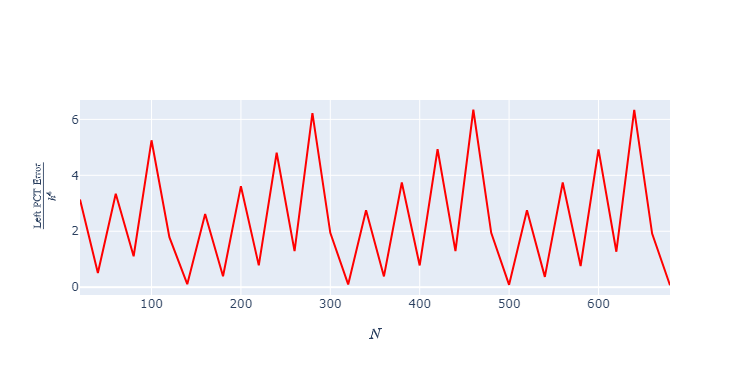}
    \caption{The error in the approximation of the left PCT divided by $\Delta^{4}$, as a function of the number of interpolation points ($N$) for the set-valued function $F$. Here, $\Delta=\frac{b-a}{N-1}$.}
    \label{fig:pct_error_example_B1}
\end{figure}

Figure \ref{fig:example_B1} correspond to the above SVF,  consists of four sub-figures. The first sub-figure shows the graph of the original function with the close-up view area near the left PCT bounded by a red rectangle. The last three sub-figures show a zoomed in view of three interpolants corresponding to different numbers of interpolation points $N$. Each interpolant is represented by vertical blue lines, drawn on the graph of the original function, which is colored in yellow.

We show the rate of decay of the interpolation error for the above SVF in figures \ref{fig:maximum_error_example_B1}. The error is measured by
$$
\text{Maximum Error}=\max_{j}{\Big\{d_{H}\big(F(\xi_{j}),\Tilde{F}(\xi_{j})\big)\Big\}},
$$
where $\{\xi_{j}\}_{j=1}^{400}$ is a set of $400$ equidistant points in $[a,b]$. We plot 
$$
G=\frac{\text{Maximum Error}}{\Delta^{4}},
$$
with $\Delta=\frac{b-a}{N-1}$, as a function of the number of the interpolation points $N$.

Finally, we show the rate of decay of the error of approximating the left PCT of the above SVF in figures \ref{fig:pct_error_example_B1}. We plot the value of 
$$
\tilde{E}=\frac{E_{k}}{\Delta^{4}},
$$
as a function of the number of the interpolation points $N$. Here, $E$ is the error of the approximation of the left PCT as in (\ref{eq:B_pct_error_expression}).
\subsubsection{Conclusions from the figures}
\begin{enumerate}
\item Figure \ref{fig:example_B1} demonstrates that the PCT approximation error decreases as $N$ increases in accordance with the theoretical rate proved in Proposition \ref{prop:pct_estimate_b2}.
\item  As seen from figures \ref{fig:maximum_error_example_B1} and \ref{fig:pct_error_example_B1}, the maximal error and the PCT approximation error both divided by $\Delta^{4}$ are bounded as $N$ increases. This indicates that the errors decays at the rate as predicted by Propositions \ref{prop:pct_estimate_b2} and \ref{prop:error_estimate_b}. 
\end{enumerate}

\subsection{Error analysis - $M$ holes}\label{Mholes}
For $F\in\mathcal{F}([a,b],M)$ (Defined in Section \ref{pre_set_svf}),
the SVF approximation $\tilde F(x)$ is defined as follows: 

For each hole $H_i$ we apply the approximation algorithm described in Section \ref{Algorithm5} for the case of one hole. The outcome includes approximations $\tilde h_i\sim h_i$, $\tilde g_i\sim g_i$ on an interval $[\tilde c_i,\tilde d_i]$ approximating the interval $[c_i,d_i]$.  In order to compare between $h_i$ and $\tilde h_i$ we extend each of them to a larger interval $[c_{e,i},d_{e,i}]$, $c_{e,i}=c_i-\epsilon$, $d_{e,i}=d_i+\epsilon$. The extensions are defined as in equations (\ref{he1}), (\ref{tildehe1}), and they are denoted $h_{e,i}$ and  $\tilde h_{e,i}$. Analogously, we define extensions $g_{e,i}$ and  $\tilde g_{e,i}$ of $g_i$ and $\tilde g_i$.

\begin{equation}\label{he1}
    h_{e,i}(x) = 
     \begin{cases}
       h_i(c) &\quad x\in [c_{e,i},c_i)\\
       h_i(x) &\quad x\in [c_i,d_i]\\ 
       h_i(d) &\quad x\in (d_i,d_{e,i}]\\
     \end{cases},
\end{equation}

\begin{equation}\label{tildehe1}
    \tilde h_{e,i}(x) = 
     \begin{cases}
       \tilde h_i(\tilde c_i) &\quad x\in [c_{e,i},\tilde c_i)\\
       \tilde h_i(x) &\quad x\in [\tilde c_i,\tilde d_i]\\
       \tilde h_i(\tilde d_i) &\quad x\in (\tilde d_i,d_{e,i}]\\
     \end{cases}.
\end{equation}

By the assumptions on the holes, there exists an $\epsilon$ such that the extensions do not cross the boundaries of $Graph(F)$. Moreover, for a small enough $\Delta$ the interval $[c_{e,i},d_{e,i}]$ contains both intervals $[\tilde c_i,\tilde d_i]$ and $[c_i,d_i]$.

Similarly to the expression for $F$ in (\ref{Fatx0}),
$F$ may be re-formulated using the above extended boundary functions, as follows:
For $x\in [a,b]$ we identify all the intervals $\{[c_{e,i},d_{e,i}]\}_{i\in I(x)} $ containing $x$. If $\#I(x)=J(x)>0$, we order the corresponding boundary values $\{g_{e,i}(x)\}$, $i\in I(x)$, in ascending order, and index the relevant holes according to this ordering $\{H_{i_j}\}_{j=1}^{J(x)}.$ The set $F(x)$ is then re-expressed as
\begin{equation}\label{Fatxe4}
F(x)=[\ell(x),g_{i_1}(x)]\cup\bigcup_{j=1}^{J(x)-1}[g_{i_j}(x),h_{i_{j+1}}(x)]\cup[h_{i_J}(x),u(x)].
\end{equation}

The approximation $\tilde F(x)$ is similarly defined as
\begin{equation}\label{tildeFatxe4}
\tilde F(x)=[\tilde \ell(x),\tilde g_{e,i_1}(x)]\cup\bigcup_{j=1}^{J(x)-1}[\tilde g_{e,i_j}(x),\tilde h_{e,i_{j+1}}(x)]\cup[\tilde h_{e,i_J}(x),\tilde u(x)].
\end{equation}

\begin{theorem}\label{TheoremC4}
Let $F$ be an SVF such that $Graph(F)$ has separable $M$ holes with $C^{4}$ boundaries. Defining the approximation by (\ref{tildeFatxe4}), for a small enough $\Delta$,
\begin{equation}
d_H(F(x),\tilde F(x))\le C\Delta^4.
\end{equation}
\end{theorem}
\begin{proof}
Each interval in (\ref{Fatxe4}) has a corresponding interval in (\ref{tildeFatxe4}), and by Proposition \ref{prop:error_estimate_b} it is clear that the Hausdorff distance between corresponding intervals is of order
$O(\Delta^4)$ as $\Delta\to 0$.
The proof is completed using the result in Lemma \ref{Lemma:sets}.
\end{proof}

\section{Set-Valued Functions with Holes of H\"older Type}\label{Holderboundaries}
\subsection{Introduction}
So far we dealt with SVFs of Lipschitz type, i.e., with holes defined by Lipschitz continuous boundary functions (see Theorem \ref{th:lip_svf}). In this section, we extend
our algorithm to deal with SVFs whose holes have upper and lower boundaries of H\"older type with H\"older exponent $\frac{1}{2}$ at both PCTs. Here again, we present the computation procedure and the error analysis for the case of an SVF with one hole defined on the interval $[c,d]\subset[a,b]$. Finally, we show how to deal with the case of several holes. We assume the hole is defined as the interior of a closed boundary curve $\Gamma\in C^{2k},\ k\in \mathbb{N}$, such that every vertical cross-section at $x\in(c,d)$ cuts the curve at two points as in Figure \ref{fig:example_C}. We further assume that $\Gamma$ has non-zero curvature at both PCTs. As before, we let the upper and lower boundaries of the hole be defined by the functions $h$ and $g$ respectively.

In the next lemma we examine the behavior of $h$ and $g$ near the left PCT, namely at $(c,h(c))$.

\begin{definition}{\bf Local series approximation}
\hfill

\medskip
Let $\{\alpha_j\}_{j=0}^J$, $\alpha_0\ge 0$, be an increasing real sequence. We say that $\sum_{j=0}^Ja_j(x-c)^{\alpha_j}$ is a local series approximation of $f(x)$ at $x=c$ if
$$|f(x)-\sum_{j=0}^ka_j(x-c)^{\alpha_j}|=o(|x-c|^{\alpha_{k}}),\ \ \text{as}\ x\to c.$$
\end{definition}

The local series approximation concept is compatible with the asymptotic expansion concept \cite{Dingle}

\begin{lemma}{\bf Local series approximations of $h$ and $g$}
\hfill

\medskip
Consider a hole with a $C^{2k}$ boundary $\Gamma$ and assume $\Gamma$ has non-zero curvature at the PCT's of the hole. Then $h$ and $g$ have a local series in powers of $(x-c)^{0.5}$.
\end{lemma}

\begin{proof}
W.l.o.g., we assume $c=0$ and $h(c)=0$.
Reflecting $\Gamma$ about the line $y=x$, and using the non-zero curvature assumption at the PCT, it follows that the inverse function $h^{-1}$ has the local power series expansion:
$$h^{-1}(y)\sim \sum_{j=2}^{2k} a_jy^j, \ a_2> 0.$$
Let $\psi(y)=\sqrt{h^{-1}(y)}$, then $\psi(0)=0$ and $\psi'(0)=\sqrt{a_{2}}>0$. By the series reversion formula (see  Abramowitz and Stegun 1972, p. 16 \cite{AbStegun}), $\psi(y)$ is invertible in a neighborhood of $y=0$, and it has a local power series expansion of the form
$$\psi^{-1}(x)=\sum_{j=1}^{2k} b_jx^j .$$
Using the relation $(u\circ v)^{-1}=v^{-1}\circ u^{-1}$, with $u(t)=sqrt(t)$ and $v=h^{-1}$, we obtain
$$\psi^{-1}(x)=h(x^2),$$
which implies, for $x\ge 0$,
$$h(x)=\psi^{-1}(\sqrt{x})=\sum_{j=1}^{2k} b_j x^{j/2} .$$

\end{proof}

Altogether, it follows that
that $h,g\in C^{2k}(c,d)$. With local expansions near $c$ and $d$ of the form, 
\begin{equation}\label{sqrtexp1}
h(x)\sim h_c^{[2k]}(x)=\sum_{j=0}^{2k} c_j (x-c)^{j/2}, \quad \text{ as } x \to c^{+}, 
\end{equation}
and
\begin{equation}\label{sqrtexp2}
h(x)\sim h_d^{[2k]}(x)=\sum_{j=0}^{2k} d_j |x-d|^{j/2}, \quad \text{ as } x \to d^{-}.
\end{equation}

\begin{corollary}\label{Coro51}
\begin{equation}\label{c2kclosedcd}
\hat h^{[2k]}\equiv h- h_c^{[2k]}-h_d^{[2k]}\in C^{k}[c,d].
\end{equation}
\end{corollary}
Similar local expansion is assumed for $g$.
For example, circular or elliptic holes fulfill these conditions. 
In this section, we develop an algorithm for constructing high order approximations to SVFs with holes of the above type. We remark here that the algorithm suggested in \cite{Levin1986} fails in approximating such holes in the neighborhood of the PCTs. 

{\bf The approximation procedure} starts with deriving a high order approximation to the PCTs.
Next, this information is used for computing local approximations of the form (\ref{sqrtexp1}) for the upper and lower boundary functions $g$ and $h$ near the PCTs. Afterwards, in view of (\ref{c2kclosedcd}), we subtract  these local approximations in order to regularize the given data of $h$ and $g$. Finally, a spline approximation is applied to the regularized data, and the final approximation is obtained by adding those previously subtracted local approximations at the PCTs. 
\begin{figure}[ht]
    \centering{{\includegraphics[width=5cm]{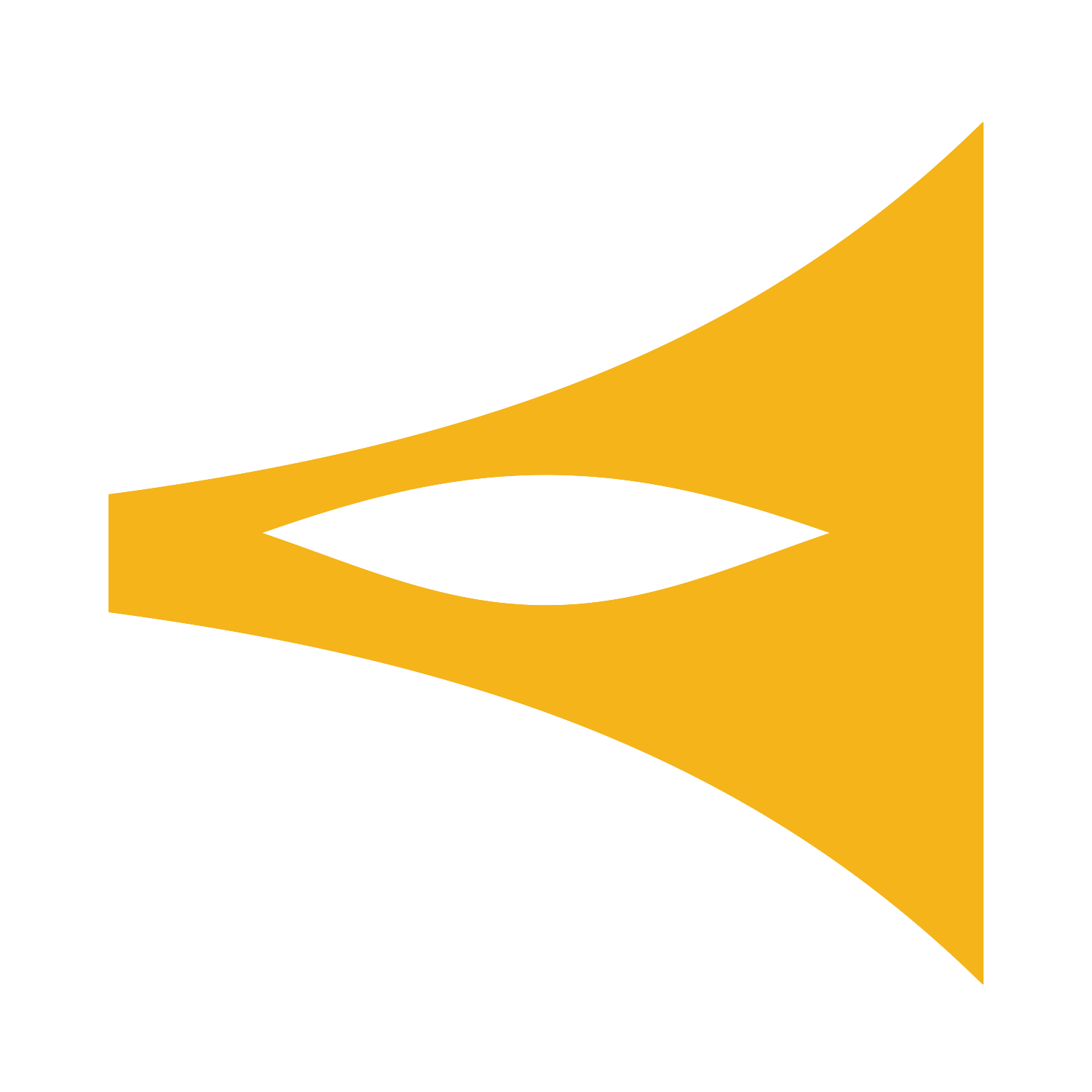} }}
    \qquad{{\includegraphics[width=5cm]{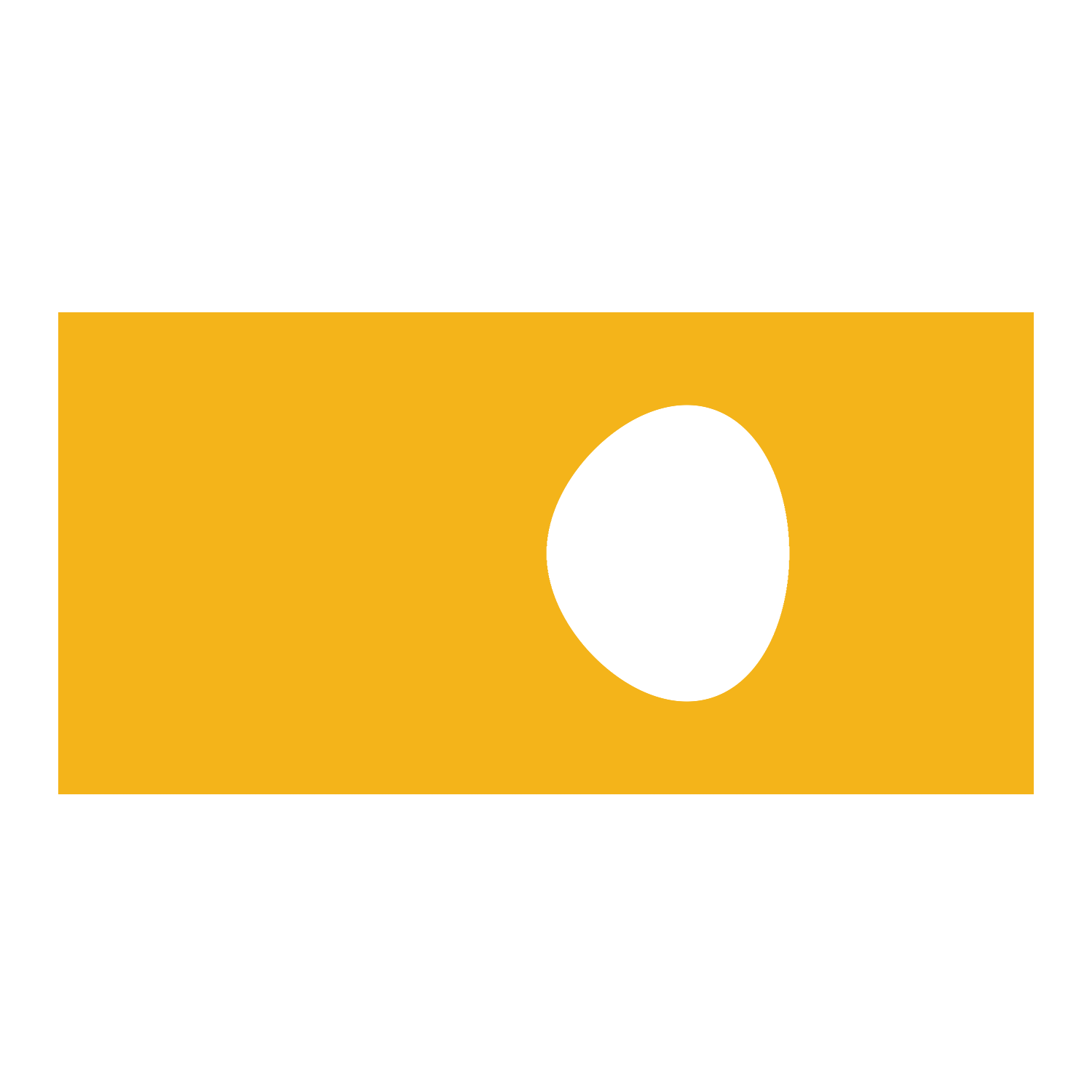} }}
    \caption{An illustration of two SVFs one with a Ho\"lder type hole (right) and one with a Lipschitz type hole (left).}%
    \label{fig:example_C}
\end{figure}
\subsection{A Description of The Algorithm}\label{Algorithm5}
We consider a set-valued function $F$ with only one hole $H$, and samples given at equally spaced points $X=\big\{{x_i}\big\}_{i=0}^{N}$, where $x_{i}=a+i\Delta,\ \Delta=\frac{b-a}{N}$. As in the previous section, we let $\{x_i\}_{i=n}^{m}$ be all the sample points in $(c,d)$. We assume that $m-n > 2k$. 

The first two steps of the algorithm are identical to the corresponding steps in the description of the algorithm of the previous chapter.
\begin{enumerate}
\item \textbf{Approximating the functions $u$ and $\ell$ describing the lower and the upper boundaries of the graph of $F$}

As in the previous chapter.

\item \textbf{Identifying the hole}

As in the previous chapter.
 \item \textbf{Approximating the right and left PCTs of $H$}:
    
In order to approximate the location of the left PCT, we swap between the $x-$coordinate and $y-$coordinate of the data points near the PCT by reflecting the graph of the hole across the line $y=x$ (see Figure \ref{fig:reflection}). Afterwards, we find a polynomial $p_{2k-1}(y)$, which interpolates the points 
\begin{align*}
    \Big\{\big(g(x_{i}),x_{i}\big)\Big\}_{j=n}^{n+k-1}\bigcup\Big\{\big(h(x_{i}),x_{i}\big)\Big\}_{j=n}^{n+k-1}.
\end{align*}
\begin{figure}[ht]
    \centering
    \subfloat[\centering Left side of a hole]{{\includegraphics[width=5cm]{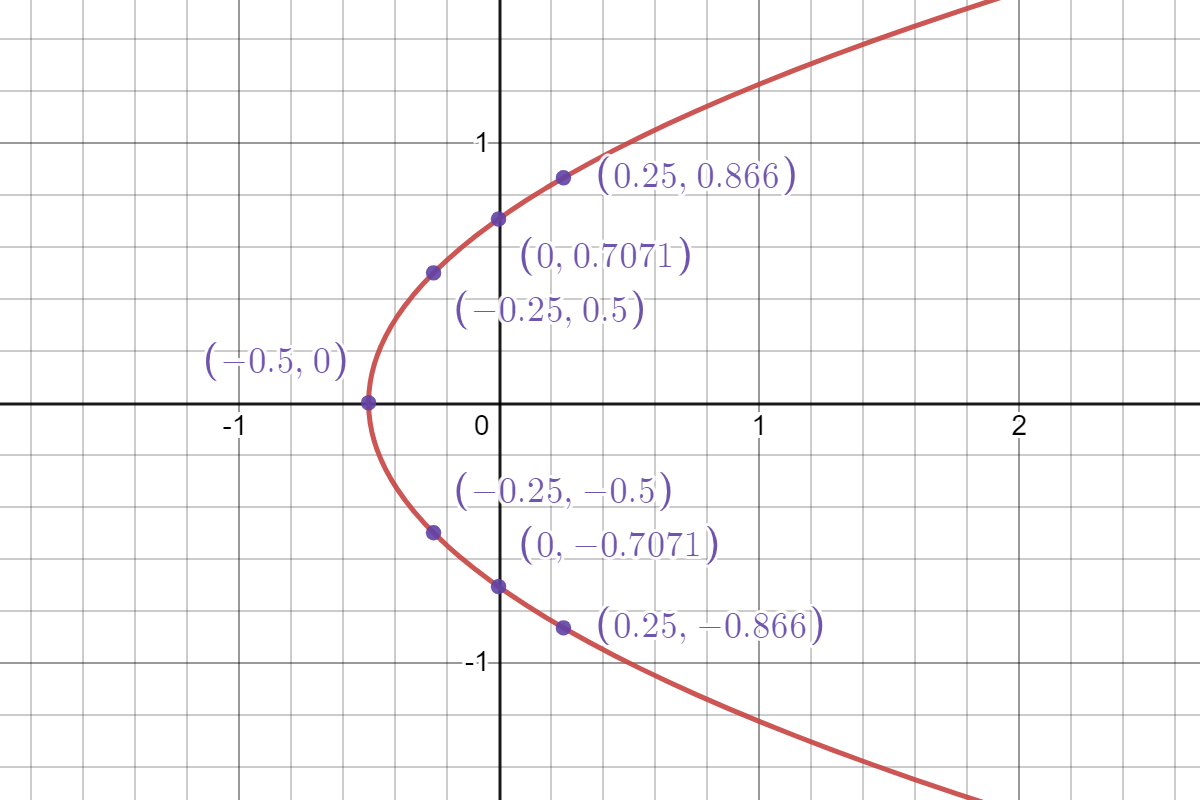}}}
    \qquad
    \subfloat[\centering Left side of a hole after the reflection across $y=x$]{{\includegraphics[width=5cm]{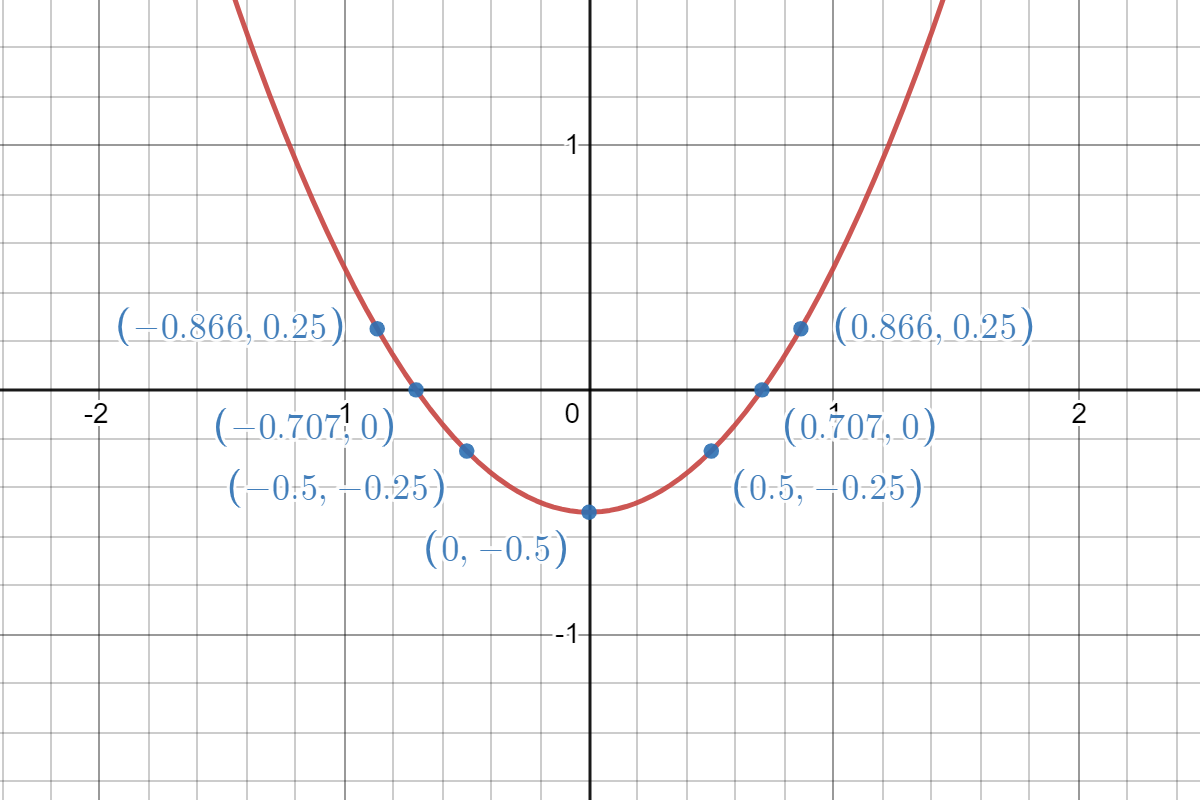}}}
    \caption{An illustration of swapping between the $x-$coordinate and $y-$coordinate of the data points of the left side of a hole, which is equivalent to reflection across the line $y=x$.}%
    \label{fig:reflection}
\end{figure}
Next, we find the minimum point $\big(\Tilde{y}_{p},p_{2k-1}(\Tilde{y}_{p})\big)$ of $p_{2k-1}(y)$ over the interval $[g(x_{n+k-1}),h(x_{n+k-1})]$. Finally, we define the approximation of the left PCT $\big(c,g(c)\big)$ to be $(p_x,p_y)\equiv \big(p_{2k-1}(\Tilde{y}_{p}),\Tilde{y}_{p}\big)$. For the right PCT we do a similar procedure, defining the approximation of the right PCT as $(q_x,q_y)$, using the maximum point of an analogue polynomial interpolating the reflected data near the right PCT.

\item \textbf{Approximating the lower and upper boundaries of the hole}:
    
In this step, we build a function that approximates the upper boundary function $h(x)$ (a similar procedure is applied for the lower boundary). The function $h(x)$, $x\in [c,d]$, is known to be $C^{2k}$ in $(c,d)$, with singularities at $c$ and $d$ of the form (\ref{sqrtexp1}) and (\ref{sqrtexp2}). The approximation procedure suggested here is based upon the observation (\ref{c2kclosedcd}), implying that $\hat h^{[2k]}$ can be efficiently approximated using spline interpolation over $[c,d]$.

In our problem we do not know the expansions $h_c^{[2k]}$ and $h_d^{[2k]}$, In particular, we do not know $c$ and $d$. Instead, 
We find two approximations $P\sim h_c^{[2k]}$ and
$Q\sim h_d^{[2k]}$. Following the singularity behavior in (\ref{sqrtexp1}) and (\ref{sqrtexp2}) we look for $P$ and $Q$ of the form

\begin{equation}\label{sqrtexp3}
P(x)=\sum_{j=0}^{r} p_j (x-p_x)^{j/2},
\end{equation}
\begin{equation}\label{sqrtexp4}
Q(x)=\sum_{j=0}^{r} q_j |x-q_x|^{j/2},
\end{equation}
 where $P$ interpolates the set of points 
\begin{equation}
\label{eq:first_data_set}
    \Big\{\big(p_{x},p_{y}\big)\Big\}\bigcup\Big\{\big(x_{j},h(x_{j})\big)\Big\}_{j=n}^{n+r-1},
\end{equation}
and $Q$ interpolates the set of points 
\begin{equation}
\label{eq:second_data_set}
    \Big\{\big(x_{j},h(x_{j})\big)\Big\}_{j=m-r+1}^{m}\bigcup\Big\{\big(q_{x},q_{y}\big)\Big\}.
\end{equation}

    Here $r$ is a free parameter to be determined according to the desired approximation order.
    
    Afterwards, we compute a "not-a-knot" cubic spline $S(x)$ interpolating the following set of data points
    \begin{equation}\label{Rdata}
        \quad\Big\{\big( p_{x},p_{y}-R(p_{x})\big)\Big\}
        \bigcup\Big\{\big(x_{i},h(x_{i})-R(x_{i})\big)\Big\}_{i=0}^{m}
        \bigcup\Big\{\big( q_{x},q_{y}-R(q_{x})\big)\Big\},
    \end{equation}
    where $R(x)=P(x)+Q(x)$. As explained above, the subtraction of $P$ and $Q$ intends to eliminate the singularity behavior of $h$ near the PCTs. 
    
    Finally, the upper boundary of the hole is approximated by the function
    \begin{equation}
        U(x)=S(x)+P(x)+Q(x).
   \end{equation}
   In a similar way we compute the approximation of the lower boundary of the hole.
    \end{enumerate}

    These approximations, together with the approximation of the upper and the lower boundaries of the graph of $F$, define the final approximation $\tilde F$ of the set-valued function $F$.

\subsection{Error Analysis}\label{EA}

The error analysis is composed of four steps:
\begin{itemize}
    \item Estimating the error in approximating the location of the PCT $(c,h(c))$.
    \item Bounding the error involved in the approximations $P\sim h_c^{[2k]}$ and $Q\sim h_d^{[2k]}$.
    \item Estimating the error in approximating the values $\{\hat h^{[2k]}(x_i)\}$ by $\{h(x_i)-R(x_i)\}$.
    \item Bounding the spline approximation error and combining all error estimates.
\end{itemize}

\subsubsection{ Error Analysis of the Approximation of the PCTs}\label{PCTappr}
\hfill

\medskip
In this section we find the approximation order in approximating the left PCT $\big(c,g(c)\big)$. A similar result applies for the approximation of the right PCT. 

The polynomial $p_{2k-1}$ defined in Step 3 of the algorithm is using the data at $x_i=a+i\Delta$, $n\le i \le n+k-1$. Since $(c,g(c))$ is the left PCT, it turns out that for a small enough $\Delta$, $g$ and $h$ are invertible over the interval $[c,x_{n+k-1}]$.

We define the following function
\begin{align*}
    \phi(y) = 
     \begin{cases}
       g^{-1}(y), &\quad y\in\big[g(x_{n+k-1}), g(c)\big)\\
       h^{-1}(y), &\quad y\in\big[h(c),h(x_{n+k-1})\big]\\
     \end{cases},
\end{align*}
recalling $g(c)=h(c)$. We observe that $\phi\in C^{2k}[g(x_{n+k-1}),h(x_{n+k-1})]$ since $\Gamma\in C^{2k}$. By definition, the polynomial $p_{2k-1}(y)$ interpolates $\phi(y)$ over the interval $\big[g(x_{n+k-1}),h(x_{n+k-1})\big]$. Since $h$ and $g$ has the local expansions (\ref{sqrtexp1}) and (\ref{sqrtexp2}), it follows that the maximal distance between the interpolation points for $p_{2k-1}(y)$ can be estimated as 
$\Delta_{y}=O\big(\sqrt{\Delta}\big)$ as $\Delta\to 0$.

 We note that by the definition of the inverse function, we have  $\phi(y_{p})=c$ where $y_p=h(c)=g(c)$. Thus, we denote the  actual left PCT by $\big(\phi(y_{p}),y_{p}\big)$. Recall that the approximated left PCT is $(p_x,p_y)\equiv \big(p_{2k-1}(\Tilde{y}_{p}),\Tilde{y}_{p}\big)$. In the following we estimate the Euclidean distance between the original and approximated PCT, 
\begin{equation}
    \label{eq:error_pct_approx_c}
    E_{k}=\sqrt{\big(\phi(y_{p})-p_{2k-1}(\Tilde{y}_{p})\big)^{2}+\big(y_{p}-\Tilde{y}_{p}\big)^{2}}.
\end{equation}
\begin{proposition}\label{Prop5}
Using the above algorithm for approximating the left PCT , $E=O\big(\Delta^{k-0.5}\big)$ as $\Delta\to 0$.
In particular, $|c-p_x|=O(\Delta^k)$
and $|h(c)-p_y|=O(\Delta^{k-1/2})$.
\end{proposition}
\begin{proof}
By the assumption on non-zero curvature of $\Gamma$ at the PCT, it follows that for a small enough $\Delta$ $\phi''(y)>C>0$ for $\forall y\in\big[g(x_{n+k-1}),h(x_{n+k-1})\big]$. By the mean value theorem
\begin{equation}
\label{al:error_estimate_phi_tag_c}
\frac{|\phi'(y_{p})-\phi'(\Tilde{y}_{p})|}{|y_{p}-\Tilde{y}_{p}|}=\frac{|0-\phi'(\Tilde{y}_{p})|}{|y_{p}-\Tilde{y}_{p}|}=|\phi''(y_{c})|>C>0,
\end{equation}
where $y_{c}$ is between $y_{p}$ and $\Tilde{y}_{p}$. Using the estimate for the error in approximating the derivative by polynomial interpolation, and recalling $p_{2k-1}'(\Tilde{y}_{p})=0$, we obtain
\begin{align*}
    \frac{|\phi'(\Tilde{y}_{p})|}{|y_{p}-\Tilde{y}_{p}|}=\frac{|p_{2k-1}'(\Tilde{y}_{p})+O(\Delta_{y}^{2k-1})|}{|y_{p}-\Tilde{y}_{p}|}=\frac{|0+O(\Delta_{y}^{2k-1})|}{|y_{p}-\Tilde{y}_{p}|}=\frac{O(\Delta_{y}^{2k-1})}{|y_{p}-\Tilde{y}_{p}|},
\end{align*}
and using (\ref{al:error_estimate_phi_tag_c}) we obtain
\begin{equation}
\label{eq:pct_x_error_c}
    |y_{p}-\Tilde{y}_{p}|=O(\Delta_{y}^{2k-1}).
\end{equation}
Moreover, by using the Mean Value Theorem, the interpolation error estimate, and $(\ref{eq:pct_x_error_c})$
\begin{equation}
\label{eq:pct_y_error_c1}
    |\phi(y_{p})-p_{2k-1}(\Tilde{y}_{p})|\leq |\phi(y_{p})-\phi(\Tilde{y}_{p})|+O(\Delta_{y}^{2k})= |\phi'(\xi)||y_{p}-\Tilde{y}_{p}|+O(\Delta_{y}^{2k}),
\end{equation}    
where $\xi$ is between $y_{p}$ and $\Tilde{y}_{p}$. Since $\phi'(y_p)=0$, $|\phi'(\xi)|\le C|y_{p}-\Tilde{y}_{p}|$. Hence,
\begin{equation}
\label{eq:pct_y_error_c2}
|\phi(y_{p})-p_{2k-1}(\Tilde{y}_{p})|\leq
 C|y_{p}-\Tilde{y}_{p}|^2+O(\Delta_{y}^{2k})=O(\Delta_{y}^{4k-2})+O(\Delta_{y}^{2k})=O(\Delta_{y}^{2k}).
\end{equation}

Recalling $\phi(y_{p})=c$ and $y_p=h(c)=g(c)$, and the notation for the approximated PCT, $(p_x,p_y)=(p_{2k-1}(\tilde y_p),\tilde y_p)$, it follows from $(\ref{eq:pct_x_error_c})$ and $(\ref{eq:pct_y_error_c2})$ that $|c-p_x|=O(\Delta^k)$
and $|h(c)-p_y|=O(\Delta^{k-1/2})$.
Combining the error estimates we conclude that
\begin{equation}
\label{eq:error_estimate_pct_c}
    E_{k}=O(\Delta_{y}^{2k-1})=O(\Delta^{k-\frac{1}{2}}).
\end{equation}
\end{proof}

\subsubsection{The error in approximating the local expansion near the PCT}
\hfill

\medskip
To analyze the approximations $P\sim h_c^{[2k]}$ and $Q\sim h_d^{[2k]}$ we employ the following interpolation lemma:

\begin{lemma}\label{lemmaT}
Let $T(s)$ be a local power series of $f\in C^{r+1}[a,b]$ at $0\in (a,b)$,
$$T(s)=\sum_{j=0}^ra_js^j,$$
and let $p_r(s)=\sum_{j=0}^rb_js^j$ interpolate the data
$\{(s_i,f(s_i))\}_{i=0}^r$, where $s_i\in [0,r\delta]$ and $\delta$ is the maximal distance between interpolation points. 
Then, for $0\le j\le r$,
\begin{equation}\label{InterpolationLemma}
|b_j-a_j|=O(\delta^{r+1-j}), \ \ {\text as}\ \delta\to 0.
\end{equation}
\end{lemma}

The result (\ref{InterpolationLemma}) follows using standard estimates for the approximation of a function and its derivatives by polynomial interpolation, applied at $s=0$.

Let $\tilde P(x)=\sum_{j=0}^{r} \tilde p_j (x-c)^{j/2}$ interpolate the function $h$
at the data set $\{c,\{x_{j}\}_{j=n}^{n+r-1}\}.$
We recall that $h(x)$ has a local series expansion of the form
$$h_c^{[2k]}(x)=\sum_{j=0}^{2k} c_j (x-c)^{j/2}.$$
Setting $s=(x-c)^{1/2}$, the problem is transformed into approximating $h(s^2+c)\sim \sum_{j=0}^{2k} c_j s^{j}$ by polynomial interpolation at the points
$s_0=0$ and $\{s_{j-n+1}=(x_{j}-c)^{1/2}\}_{j=n}^{n+r-1}.$ Noting that the maximal distance between the interpolation points $\{s_j\}_{j=0}^r$ is $O(\Delta^{1/2})$, and using the above lemma, we obtain for $0\le j\le r$,
$$|\tilde p_j-c_j|=O(\Delta^{\frac{r+1-j}{2}}), \ \ {\text as}\ \Delta\to 0.$$
Here it also follows that $\tilde p_0=c_0=h(c)$.

Since the location of the PCT is unknown, we cannot compute $\tilde P$. Instead, we compute $P(x)=\sum_{j=0}^{r} p_j (x-p_x)^{j/2}$ by interpolating the data
\begin{equation}
    \Big\{\big(p_{x},p_{y}\big)\Big\}\bigcup\Big\{\big(x_{j},h(x_{j})\big)\Big\}_{j=n}^{n+r-1},
\end{equation}
and we need to estimate the errors $|p_j-c_j|$.

By Proposition \ref{Prop5} we have
$|c-p_x|=O(\Delta^{k})$ and $|h(c)-p_y|=O(\Delta^{k-1/2})$.

Let us compare the interpolation problems for $\tilde P$ and for $P$, using the Lagrange interpolation formula. We observe that in the denominators of the Lagrange polynomials for $P$ there is an $O(\Delta^{k-3/2})$ perturbation relative to those for $\tilde P$. It follows that $$p_0=p_y=c_0+O(\Delta^{k-\frac{1}{2})},$$
and for $1\le j\le r$
\begin{equation}\label{pjmcj}
|p_j-c_j|=O(\Delta^{\frac{r+1-j}{2}})+O(\Delta^{k-\frac{3}{2}}).
\end{equation}
Similar estimates hold for the approximation $Q(x)$.

Comparing $P$ and $h_c^{[2k]}$, the other source of discrepancy is due to different expansion points, $p_x\ne c$. The major influence comes from the leading singular terms in the power series expansions. Using the estimate $|c-p_x|=O(\Delta^k)$ as $k\to 0$, it follows that 
\begin{equation}\label{cmxp}
||x-c|^\frac{1}{2}-|x-p_x|^\frac{1}{2}|\le C|c-x_p|^\frac{1}{2}=O(\Delta^\frac{k}{2}),\ \ {\text as}\ \Delta\to 0.
\end{equation}
The contribution of other terms in the expansion is of higher order in $\Delta$. 

\subsubsection{The error in the spline approximation to $\hat h^{[2k]}$.}
\hfill

\medskip
For $k\ge s$, $s$ even, by Corollary \ref{Coro51}, $\hat h^{[2k]}\in C^s[c,d]$. 
Using a $s$ order spline interpolant, $\hat S$, to approximate $\hat h^{[2k]}$ yields a uniform approximation error
$$
|\hat h^{[2k]}(x)-\hat S(x)|\le C\Delta^s,\ \ \forall x\in [c,d].
$$
However, since $\hat h^{[2k]}$ is unavailable, we apply the spline interpolation to the data (\ref{Rdata}), which is an approximation of exact data of $\hat h^{[2k]}$.

Using the estimates (\ref{pjmcj}), (\ref{cmxp}), we have that
\begin{equation}\label{h2kmR}
|h_c^{[2k]}(x)+h_d^{[2k]}(x)-R(x)|=O(\Delta^{\frac{r}{2}+\frac{1}{2}})+O(\Delta^{k-\frac{3}{2}})+O(\Delta^{\frac{k}{2}}),\ \ x\in[\max\{c,p_x\},c+K\Delta],
\end{equation}
where $K$ is independent of $\Delta$. Using $r=4$ and $k=5$ gives
\begin{equation}\label{h2kmR2}
|h_c^{[10]}(x)+h_d^{[10]}(x)-R(x)|=O(\Delta^{\frac{5}{2}}),\ \ x\in[\max\{c,p_x\},c+K\Delta].
\end{equation}
In the following we assume $k\ge 3$ which implies that the term $O(\Delta^{k-\frac{3}{2}})$ is redundant. A similar estimate holds near the right PCT.

Approximating the PCT location using $k$ sample points, approximating the local series expansions $P$ and $Q$ using the approximated PCT and $r$ samples, approximating the resulting regularized data using a not-a-knot cubic spline interpolation $S$ on $[p_x,q_x]$, we define the final approximation to $h$ as 
\begin{equation}\label{tildeh}
\tilde h(x)=S(x)+P(x)+Q(x),\ \ x\in [p_x,q_x].
\end{equation}

We notice that $h$ is defined on the interval $[c,d]$, while $\tilde h$ is defined on $[p_x,q_x]$. In order to compare between $h$ and $\tilde h$ we extend each of them to a larger interval $[c_e,d_e]$, $c_e=p_x-\epsilon$, $d_e=q_x+\epsilon$, as follows:

\begin{equation}\label{he}
    h_e(x) = 
     \begin{cases}
       h(c) &\quad x\in [c_e,c)\\
       h(x) &\quad x\in [c,d]\\ 
       h(d) &\quad x\in (d,d_e]\\
     \end{cases},
\end{equation}

\begin{equation}\label{tildehe}
    \tilde h_e(x) = 
     \begin{cases}
       \tilde h(p_x) &\quad x\in [c_e,p_x)\\
       \tilde h(x) &\quad x\in [p_x,q_x]\\
       \tilde h(q_x) &\quad x\in (q_x,d_e]\\
     \end{cases}.
\end{equation}

Using the estimates in Proposition \ref{Prop5}, $|c-p_x|$ and $|d-q_x|$ are both bounded by $C\Delta^k$. Hence, choosing $\epsilon\ge C\Delta^k$ guarantee that $[c_e,d_e]$ contains both intervals $[c,d]$ and $[p_x,q_x]$.

All the above estimates lead to the following approximation theorem:

\begin{proposition}\label{Prop6}
Using spline interpolation of order $s$ for $S$, and assuming $k\ge 3$,
\begin{equation}\label{tildehapp}
\|h_e-\tilde h_e\|_{[a,b],\infty}=O(\Delta^{\frac{r}{2}+\frac{1}{2}})+O(\Delta^{s})+O(\Delta^{\frac{k}{2}}), {\ \ \text as}\ \ \Delta\to 0.
\end{equation}
\end{proposition}

A similar construction, with a similar approximation estimate hold for the approximation $\tilde g_e$ to $g_e$.

\subsubsection{The case of $M$ holes}\label{Mholes5}
\hfill

\medskip
Let $F$ be an SVF such that $Graph(F)$ has separable $M$ holes $\{H_i\}_{i=1}^M$ (i.e. the closures of the holes are disjoint). The hole $H_i$ is defined on an interval denoted by $[c_i,d_i]\subset(a,b)$, and we assume that it is simple, namely, it is defined as the interior of a closed boundary curve $\Gamma_i$, such that every vertical cross-section at $x\in(c_i,d_i)$ cuts the curve at two points. We further assume that the curves $\{\Gamma_i\}$ do not intersect each other, and do not intersect the upper and the lower boundaries of $Graph(F)$. We further assume that each $\Gamma_i$ has non-zero curvature at both PCTs of $H_i$. Let the upper and lower boundaries of $H_i$ be defined by the functions $h_i$ and $g_i$ respectively. We also recall the functions $u$ and $\ell$ defining the upper and lower boundaries of $Graph(F)$.

The SVF approximation $\tilde F(x)$ is defined as follows: For each hole $H_i$ we apply the approximation algorithm described in Section \ref{Algorithm5} for the case of one hole. The outcome includes approximations $\tilde h_i\sim h_i$, $\tilde g_i\sim g_i$ on an interval $[\tilde c_i,\tilde d_i]\equiv [p_{x,i},q_{x,i}]$ approximating the interval $[c_i,d_i]$. 

We continue the analysis as in Section \ref{Mholes}, using the same idea of extended functions on extended intervals, and using the same definitions therein.

\begin{theorem}\label{Theorem5}
Let $F$ be an SVF such that $Graph(F)$ has separable $M$ holes with $C^{2k}$ boundary curves with non-zero curvatures at the PCTs. Defining the approximation by (\ref{tildeFatxe4}), then, for a small enough $\Delta$,
\begin{equation}
d_H(F(x),\tilde F(x))\le C_1\Delta^{\frac{r}{2}+\frac{1}{2}}+C_2\Delta^{s}+C_3\Delta^{\frac{k}{2}}.
\end{equation}
\end{theorem}
\begin{proof}
Each interval in (\ref{Fatxe4}) has a corresponding interval in (\ref{tildeFatxe4}), and by (\ref{tildehapp}) it is clear that the Hausdorff distance between corresponding intervals is of order
$$O(\Delta^{\frac{r}{2}+\frac{1}{2}})+O(\Delta^{s})+O(\Delta^{\frac{k}{2}}), {\ \ \text as}\ \ \Delta\to 0.$$
The proof is completed using the result in Lemma \ref{Lemma:sets}.
\end{proof}

\subsection{Numerical Results}
We demonstrate the process of approximating the left PCT as well as showing the decay rate of the interpolation error on a set-valued function with an elliptic hole, displayed in figure \ref{fig:example_C1}, which is explicitly given by,

\[ F(x)=\begin{cases} 
      \big[-\frac{3}{2},\frac{3}{2}\big],\qquad &x\in\big[-1,1\big]/\big[-\frac{1}{2},\frac{1}{2}\big],\\
       \Big[-\frac{3}{2}, -\sqrt{1-4x^2}\Big]\bigcup\Big[\sqrt{1-4x^2}, \frac{3}{2}\Big],\qquad &x\in\big[-\frac{1}{2}   ,\frac{1}{2}\big].
   \end{cases}
\]

\begin{figure}[!ht]
  \centering
  \subfloat[][The Set-Valued Function $F$]{\includegraphics[width=.4\textwidth]{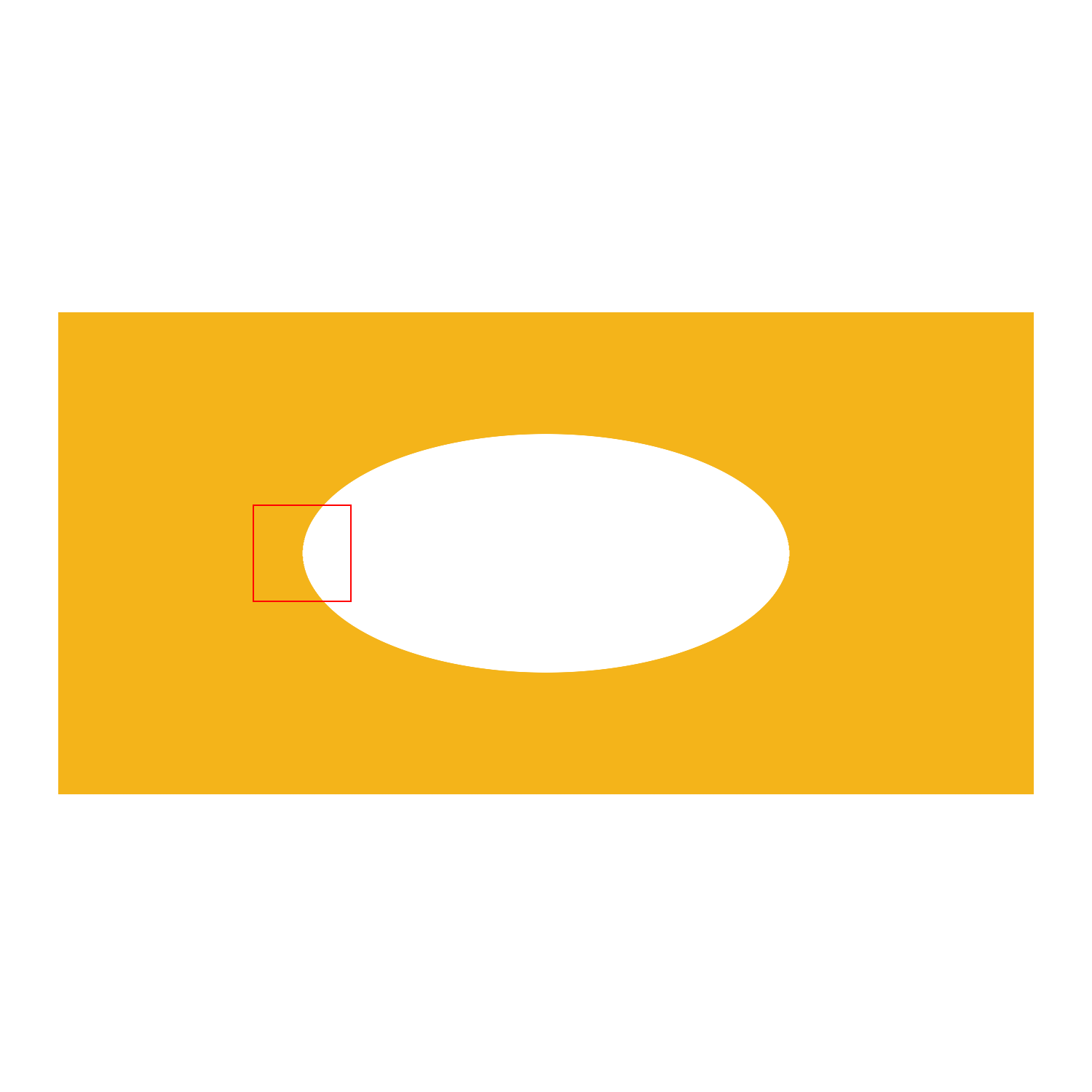}}\quad
  \subfloat[][Approximation with 30 samples]{\includegraphics[width=.4\textwidth]{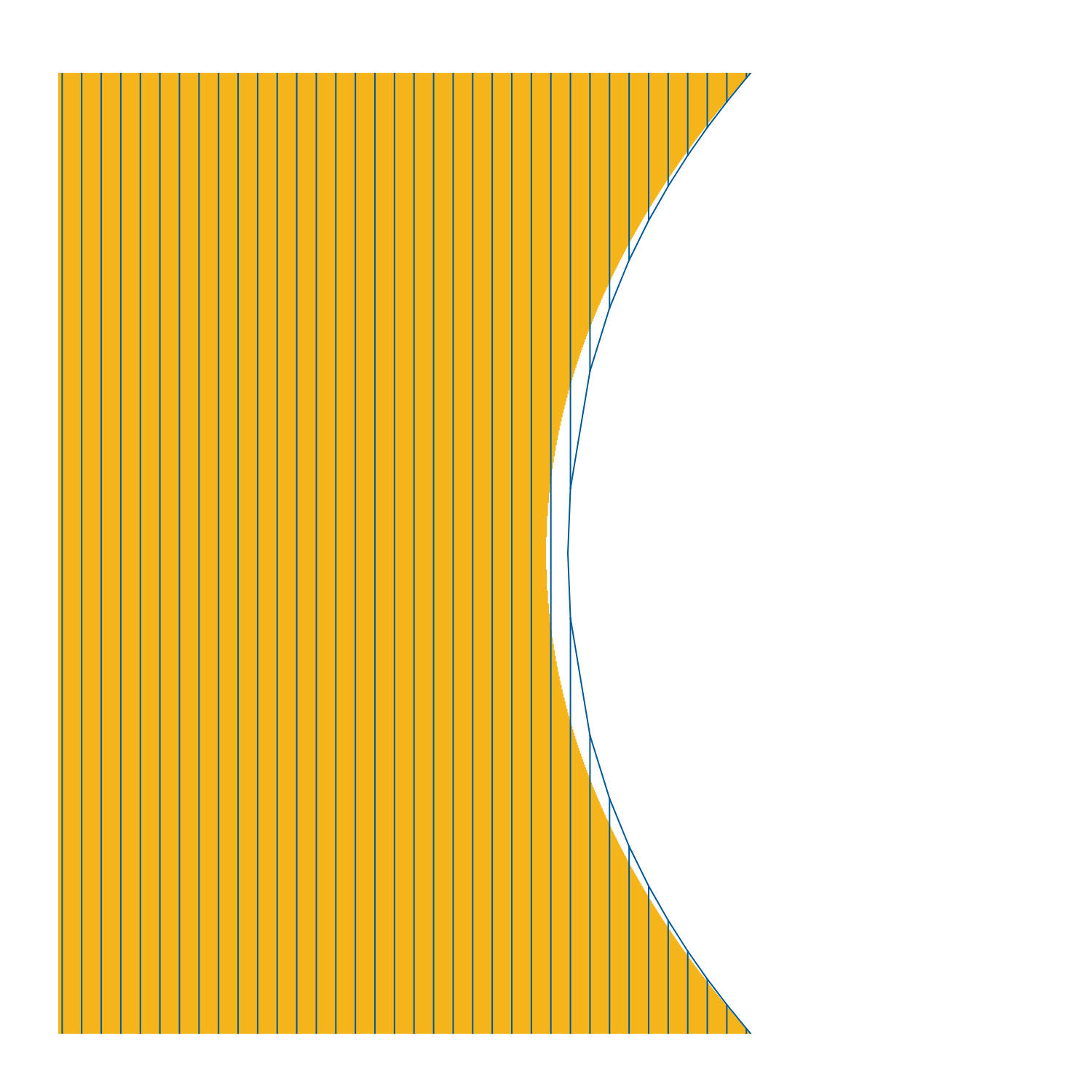}}\\
  \subfloat[][Approximation with 35 samples]{\includegraphics[width=.4\textwidth]{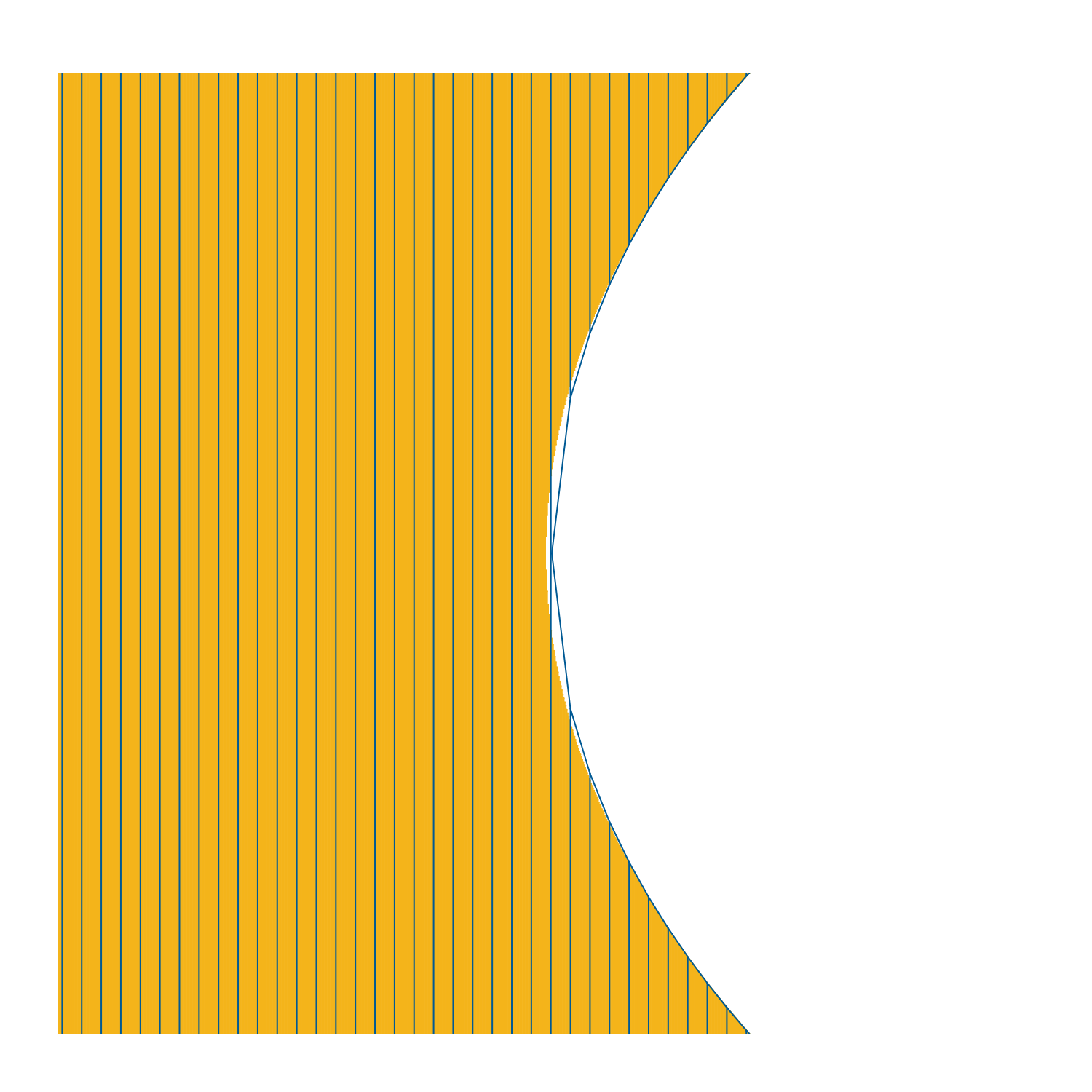}}\quad
  \subfloat[][Approximation with 40 samples]{\includegraphics[width=.4\textwidth]{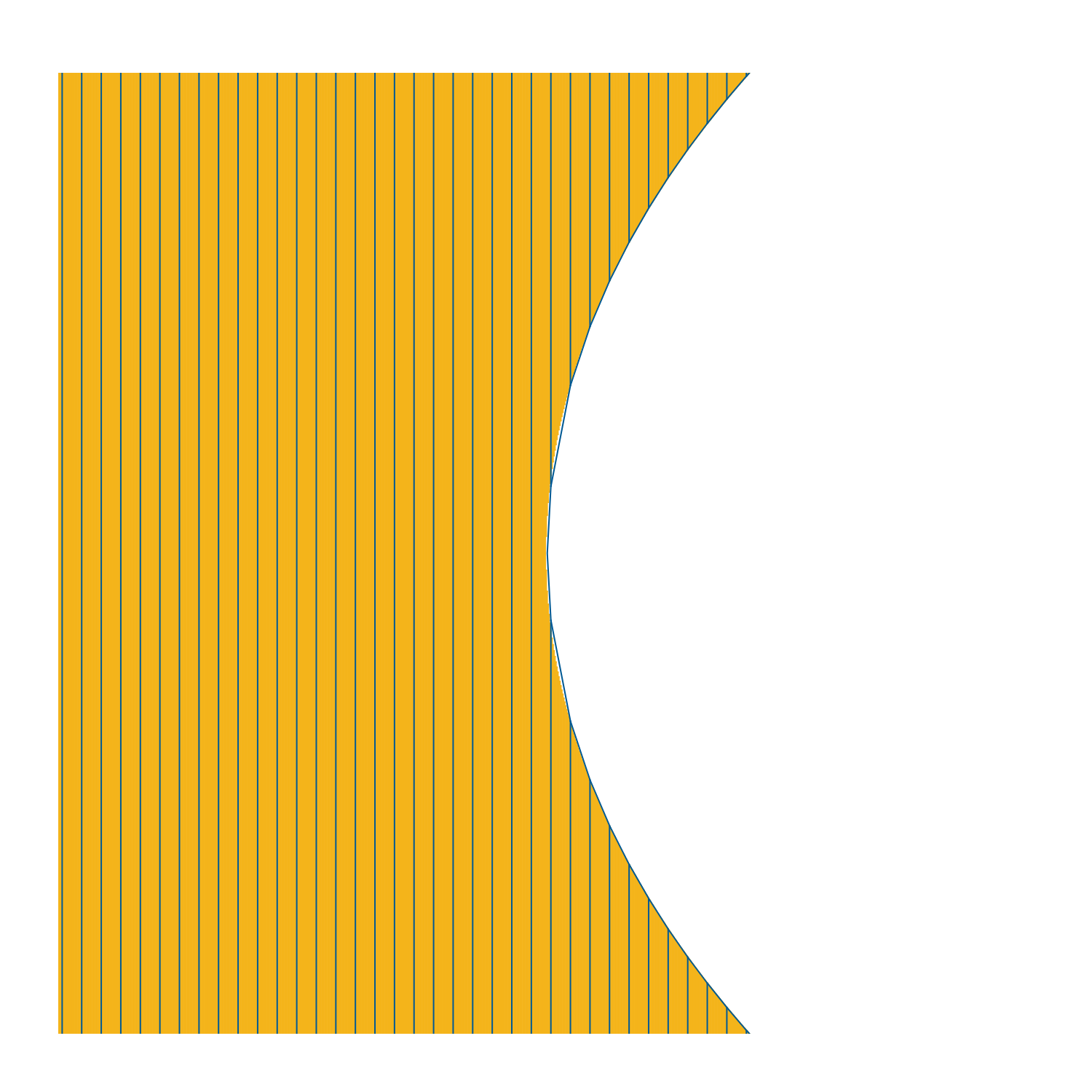}}
  \caption{The set-valued function $F$ and its approximations, zoomed in near the left PCT of the hole. Each approximation is represented by vertical blue lines, drawn on the graph of the  original function, which is colored in yellow.}
  \label{fig:example_C1}
\end{figure}
\begin{figure}[!ht]
    \centering
    \includegraphics[width=0.8\textwidth]{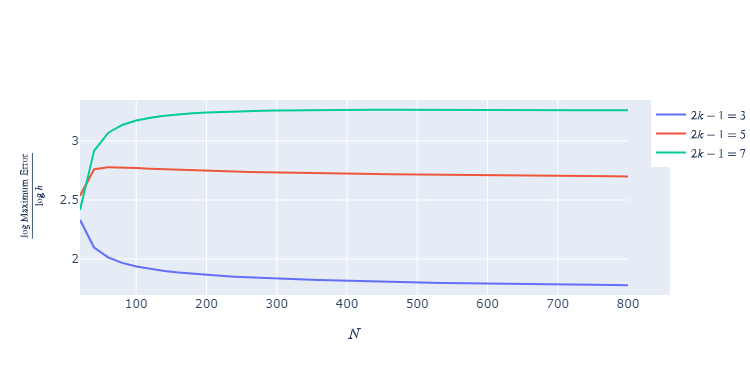}
    \caption{The rate of decay of the interpolation error as a function of the number of interpolation points ($N$) for the set-valued function $F$ and for three values of $k$, $s=3$ and $r=4$.}
    \label{fig:maximum_error_example_C1}
\end{figure}
\begin{figure}[!ht]
    \centering
    \includegraphics[width=0.8\textwidth]{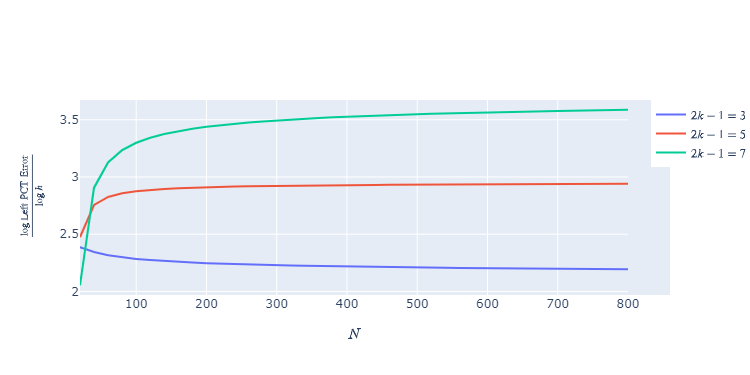}
    \caption{The rate of decay of the error in the location of the left PCT as a function of the number of interpolation points ($N$) for the set-valued function $F$ and for three values of $k$, $s=3$ and $r=4$.}
    \label{fig:pct_error_example_C1}
\end{figure}
Figure \ref{fig:example_C1} consists of four sub-figures. The first sub-figure shows the graph of the original function with the close-up view area near the left PCT bounded by a red rectangle. The last three sub-figures show a zoomed-in view of the approximations corresponding to a different number of samples. The approximant values are represented by vertical blue lines, drawn on the graph of the original function, which is colored in yellow.

We show the rate of decay of the approximation error for the above SVF in figure \ref{fig:maximum_error_example_C1}. The error is estimated by
$$
\text{Maximum Error}=\max_{j}{\Big\{d_{H}\big(F(\xi_{j}),\Tilde{F}(\xi_{j})\big)\Big\}},
$$
where $\{\xi_{j}\}_{j=1}^{400}$ is a set of equidistant points in $[a,b]$. 

Recall that the interpolating polynomial for approximating the PCT is of degree $2k-1$ (See the description of the algorithm). We plot 
$$
G_{k,r,s}=\frac{\log{(\text{Maximum Error}})}{\log{(\Delta)}},\qquad k=2,3,4,\qquad s=3,\qquad r=4,
$$
with $\Delta=\frac{b-a}{N-1}$, as a function of the number of the interpolation points $N$.

Finally, we show the rate of decay of the error of approximating the left PCT of the three above SVFs in figures \ref{fig:pct_error_example_C1}. We plot the value of 
$$
\tilde{E}_{k}=\frac{\log{(E_{k})}}{\log{(\Delta)}},\qquad k=2,3,4,
$$
as a function of the number of the interpolation points $N$. Here, $E_{k}$ is the error of the approximation of the left PCT defined in (\ref{eq:error_pct_approx_c}).

Figure \ref{fig:example_C1} demonstrates that the approximation error of the left PCT decreases as $N$ increases in accordance with the theory.
Figure \ref{fig:pct_error_example_C1} shows that the decay rate of the approximation of the left PCT improves as $k$ increases as predicted by the theoretical result (\ref{eq:error_estimate_pct_c}).

\clearpage

\appendix

\section*{Appendix A. \ \ 
Proof of Lemma~\ref{Lemma:sets}}

\noindent We state here the lemma for the convenience of the readers.

Lemma~\ref{Lemma:sets}
Let $A_1, A_2, B_1,B_2$ be subsets of $\mathbb {R}^d$. Then
$$ d_H\big(A_1\cup A_2,B_1\cup B_2\big)\le \max\big\{d_H(A_1,B_1),d_H(A_2,B_2)\big\}.$$
\begin{proof}
The proof follows directly from the definition of the Hausdorff distance between two sets in $\mathbb {R}^d$,
$$
d_H(C,D)=\max\{\sup_{c\in C}\ \inf_{d\in D}\|c-d\|,\ \ 
\sup_{d\in D}\ \inf_{c\in C}\|d-c\|\}.$$

It is easy to observe that for $a_i\in A_i,\ \ i=1,2$,
$$
\inf_{b\in B_{1}\cup B_2}\|a_i-b\|
\|\le\inf_{b_i\in B_i}
\|a_i-b_i\|,\ \ \ i=1,2.$$
Therefore
\begin{equation}\label{eq:H}
\alpha_i=\sup_{a_i\in A_i}{\ \inf_{b\in B_1\cup B_2}{\|a_i-b\|}}\le\sup_{a_i\in A_i}\ \inf_{b_i\in B_i}
\|a_i-b_i\|=\tilde{\alpha}_i,\ \ i=1,2, \end{equation}
and similarly
\begin{equation}\label{eq:G}
\beta_i=\sup_{b_i\in B_i}\ \inf_{a\in A_1\cup A_2}\|b_i-a\|\le\sup_{b_i\in B_i}\ \inf_{a_i\in A_i}
\|b_i-a_i\|=\tilde{\beta}_i,\ \ i=1,2. \end{equation}
Since  $d_H(A_i,B_i)=\max\{\tilde{\alpha}_i,\tilde{\beta}_i\}$, for $i=1,2$ we have
\begin{equation}
\label{eq:boundsonH,G}
\tilde{\alpha}_i\le d_H(A_i,B_i),\ \ \ \tilde{\beta}_i\le h_D(A_i,B_i),\   i=1,2.
\end{equation}

Now, 
$$
h_D(A_1\cup A_2, B_1\cup B_2)=\max\{\sup_{a\in A_1\cup A_2}\ \inf_{b\in B_1\cup B_2}\|a-b\|,
\sup_{b\in B_1\cup B_2}\ \inf_{a\in A_1\cup A_2}\|b-a\|\}\le\max\{\alpha_1,\beta_1,\alpha_2,\beta_2\}.
$$
By~\eqref{eq:H} and~ \eqref{eq:G} 
$\alpha_i\le \tilde{\alpha}_i$ and $\beta_i\le\tilde{\beta}_i\ $ for $i=1,2$. 
In view of \eqref{eq:boundsonH,G}, we arrive at the claim of the lemma
$$h_D(A1\cup A_2,B_1\cup B_2)\le\max\{h_D(A_1,B_1),h_D(A_2,B_2)\}.$$
\end{proof}

\section*{Appendix B. \ \  A pseudo code for the algorithm of Section \ref{sec:computed_svf_interpolant}}
\label{app1}

\begin{algorithm}
    \SetAlgoLined
    \underline{\texttt{Interpolate-Set-Valued-Function}} $(S,X)$
    \SetKwInOut{Input}{Input}
    \SetKwInOut{Output}{Output}
    
    \Input{A set of interpolation points $X=\{x_{i}\}_{i=0}^{N}$ and the samples $S=\{F(x_{i})\}_{i=0}^{N}$ at these points, where for $i=0,\ldots,N$  $F(x_{i})=\bigcup_{j=0}^{M_i}[a^{[i]}_{2j},a^{[i]}_{2j+1}]$}
    \Output{The metric polynomial interpolant $\mathcal{P}^{M}_{X}F(x)$}
    //Creating discrete samples \\
    Create $\tilde{S}=\big\{V^{[i]}\big\}_{i=0}^{N}$ with $V^{[i]}=\bigg \{ \big\{a^{[i]}_{0},a^{[i]}_{1}\big\}, \ldots, \big\{a^{[i]}_{2M_{i}},a^{[i]}_{2M_{i}+1}\big\} \bigg \}$  \;
    // See algorithm \ref{alg:Find-_Significant_Metric_Chains}\\
    $C$ = \texttt{Find-Significant-Metric-Chains}($\tilde{S}$, $X$) \;
    $\mathcal{P} \gets$ Initialise a new list \;
    \ForEach{$chain$ in $C$}
    {\textbf{p} = \texttt{Compute-Polynomial-Interpolant}$(chain,X)$) \;
    Add \textbf{p} to the list $\mathcal{P}$ \;}
    // Creating the data structure of $\mathcal{P}^{M}_{X}F(x)$ (algorithm \ref{alg:Create_Set_Valued_Interpolant_Data_Structure})\\
    DS = \texttt{Create-Set-Valued-Interpolant-Data-Structure}$(\mathcal{P},\tilde{S}, X)$ \;
    return DS \; 
    \caption{The algorithm returns a data structure of the \textit{metric polynomial interpolant} of a set-valued function from a finite number of samples. It uses a known algorithm for computing interpolating polynomials for single-valued functions. For convenience we call it \texttt{Compute-Polynomial-Interpolant}.}
    
\end{algorithm}
\begin{algorithm}
    \SetAlgoLined
    \underline{\texttt{Find-Significant-Metric-Chains}} $(\Tilde{S},X)$
    \SetKwInOut{Input}{Input}
    \SetKwInOut{Output}{Output}

    \Input{A set of interpolation points $X=\{x_{i}\}_{i=0}^{N}$ and the discrete samples $\tilde{S}=\{V^{[i]}\}_{i=0}^{N}$ derived from $S=\{F(x_{i})\}_{i=0}^{N}$}
    \Output{A set of all significant metric chains $C=\{c_{k}\}_{k=1}^{L}$}
    $T\gets$ Initialise a tree with an empty root \;
    Convert all elements in each $V^{[i]}$ to nodes \;
    //Each node contains a real value\\
    Connect all elements of $V^{[0]}$ to the root of $T$ \;
    \For{$i\gets0$ to $N-1$}{
    // See algorithm \ref{alg:Find_And_Connect_Metric_Pairs}\\
    \texttt{Find-And-Connect-Metric-Pairs}$(V^{[i]},  V^{[i+1]})$ \;}
    // See algorithm \ref{alg:Connect_PCT_Nodes_To_The_Root}\\
    \texttt{Connect-PCT-Nodes-To-The-Root}$(\tilde{S},X,T)$ \;
    $C$ = \texttt{Pre-Order-Traverse-Tree}($T$) \;
    Delete the first element in all the lists in $C$ \;
    Return C \;
    \caption{The algorithm finds all \textit{significant metric chains} from a finite number of discrete samples taken from a set-valued function. It uses a tree data structure and a known algorithm \texttt{Pre-Order-Traverse-Tree}, which generates a list of all paths of the tree. Each path represents a significant metric chain.}
    \label{alg:Find-_Significant_Metric_Chains}
\end{algorithm}
\begin{algorithm}
    \underline{\texttt{Find-And-Connect-Metric-Pairs}} $(L, R)$\;
    \SetKwInOut{Input}{Input}
    \SetKwInOut{Output}{Output}

    \Input{Two lists $L$ and $R$ of lists of nodes. Each node contains a real value}
   
    \ForEach{$\mathcal{N}$ in $L$}{
    \ForEach{$p$ in $\mathcal{N}$}{
    //Check if the node $p$ is already connected\\
    \If{$p$ has at least one child}{Jump to the next node \;}
    \If{$p$ is the minimal point of $\mathcal{N}$ and $\mathcal{N}$ is not the last set in $L$}
    {
        $\mathcal{M}\gets$ next set in $L$ \;
        $q\gets$ is the maximal point in $\mathcal{M}$ \;
        \texttt{Find-if-PCT-Can-Be-Added-and-Connect-it}($p$,$q$,$R$,$\uparrow$) \;
    }
    \ElseIf{$p$ is an interior point of $\mathcal{N}$}{
       \ForEach{$\mathcal{O}$ in R}{\If{$\min{\mathcal{O}} \leq p \leq \max{\mathcal{O}}$}{
        Create a new node $m$ with the same value as $p$ \;
        Add the node $m$ to the list $\mathcal{O}$ \;
        Connect the node $m$ to $p$ as a child \;
        Stop the loop \;
        }}
    }
    \If{$p$ has no children nodes}
    {$MP$=\texttt{Find-Closet-Nodes-in-Sample-to-Node}$(p,R)$ \;
    \ForEach{$q$ in $MP$}{
    Connect the node $q$ to $p$ as a child \;
    }}}}

    \ForEach{$\mathcal{N}$ in $R$}{
    \ForEach{$p$ in $\mathcal{N}$}{
    \If{$p$ has at least one ancestors}{Jump to the next iteration \;}
    \If{$p$ is the minimal point of $\mathcal{N}$ and $\mathcal{N}$ is not the last set in $R$}
    {
        $\mathcal{M}\gets$ next set in $R$ \;
        $q\gets$ is the maximal point in $\mathcal{M}$ \;
        \texttt{Find-if-PCT-Can-Be-Added-and-Connect-it}($p$,$q$,$L$,$\downarrow$) \;
    }
    \If{$p$ has no ancestors nodes}
    {$MP$=\texttt{Find-Closet-Nodes-in-Sample-to-Node}$(p,L)$ \;
    \ForEach{$q$ in $MP$}{
    Connect the node $q$ to $p$ as an ancestor \;
    }}}}
    \caption{The algorithm finds and connects all metric pairs between two given discrete samples. Note that the node and its value are denoted by the same notation.}
    \label{alg:Find_And_Connect_Metric_Pairs}
    
\end{algorithm}
\begin{algorithm}
    \underline{\texttt{Find-if-PCT-Can-Be-Added-and-Connect-it}} $(p,q,K,t)$\;
    \SetKwInOut{Input}{Input}
    \SetKwInOut{Output}{Output}

    \Input{A discrete sample $K$, two points $p$ and $q$ and a connection type $t$: $\uparrow$ or $\downarrow$}
    \ForEach{$\mathcal{O}$ in K}{\If{$\min{\mathcal{O}} \leq \frac{p+q}{2} \leq \max{\mathcal{O}}$}{
        Create a new node $m$ with the value of $\frac{p+q}{2}$ \;
        Add the node $m$ to the list $\mathcal{O}$ \;
        \eIf{$t$ is $\downarrow$}
        {Connect the nodes $p$ and $q$ to $m$ as children}
        {Connect the nodes $p$ and $q$ to $m$ as ancestors}
        }}
    \caption{Find whether a point of change of topology can be added to an interval in a given discrete sample, and then connect it}
\end{algorithm}
\begin{algorithm}
    \underline{\texttt{Find-Closet-Nodes-in-Sample-to-Node}} $(p,K)$\;
    \SetKwInOut{Input}{Input}
    \SetKwInOut{Output}{Output}

    \Input{A discrete sample $K$ and a node $p$}
    \Output{All nodes which are closest to $p$ from $K$}
    \ForEach{$\mathcal{S}$ in $K$}{
    \If{$\min{(\mathcal{S})}\leq p\leq \max{(\mathcal{S})}$}
    {
     \eIf{$\mathcal{S}$ contains a point equals to $p$}
     {return $p$ \;}
     {return $\emptyset$ \;}}
     }
     return arg $\min_{q\in K}{(|p-q|)}$ \;
    \caption{Find all nodes, in a given sample, whose values are the closest a to a value of a given node.}
    \label{alg:Find_if_PCT_Can_Be_Added_and_Connect_it}
\end{algorithm}
\begin{algorithm}
    \underline{\texttt{Connect-PCT-Nodes-To-The-Root}} $(\tilde{S},T)$\;
    \SetKwInOut{Input}{Input}
    \SetKwInOut{Output}{Output}

    \Input{The set of discrete samples $\tilde{S}=\{V^{[i]}\}_{i=0}^{N}$ and the Tree $T$}
   
    \For{$i\gets N$ to $0$}{
    \ForEach{$\mathcal{I}$ in $V^{[i]}$}{
    \ForEach{$p$ in $\mathcal{I}$}{
    \If{$p$ has no ancestors}{
    \eIf{$i$=0}{Connect $p$ to the root of the tree $T$ \;}{
    \ForEach{$\mathcal{O}$ in $V^{[i-1]}$}
    {\If{$\min{\mathcal{O}} \leq p \leq \max{\mathcal{O}}$}{
        Create a new node $m_{i-1}$ with the value of $p$ \;
        Add the node $m_{i-1}$ to the list $\mathcal{O}$ \;
        Connect the node $m_{i-1}$ to $p$ as an ancestor \;
        Go to the next $p$ \;
        }}
    }}}}}
    \caption{Connect all PCT nodes to the previous samples up to the root}
    \label{alg:Connect_PCT_Nodes_To_The_Root}
\end{algorithm}
\begin{algorithm}
    \SetAlgoLined
    \underline{\texttt{Create-Set-Valued-Interpolant-Data-Structure}} $(\mathcal{P},\tilde{S},X)$\;
    \SetKwInOut{Input}{Input}
    \SetKwInOut{Output}{Output}
    
    \Input{A set of interpolation points $X=\{x_{i}\}_{i=0}^{N}$ , the discrete samples $\tilde{S}=\{V^{[i]}\}_{i=0}^{N}$ at $X$ and a set of real-valued polynomials $\mathcal{P}$ interpolating the significant chains $C$ at $X$}
    \Output{A data structure that represents the set-valued interpolant $\Tilde{F}$}
    $\{\tilde{c}_i\}_{j=1}^{M},\{\tilde{d}_i\}_{j=1}^{M}\gets$ Extract-Approximated-PCTs-From-Discrete-Samples($\tilde{S}$)\;\\
    $\big\{\tilde{u},\tilde{\ell}\big\}\bigcup\big\{\tilde{g}_{j}\big\}_{j=1}^{M}\bigcup\big\{\tilde{h}_{j}\big\}_{j=1}^{M}\gets$ Identify-Boundary-Functions($\mathcal{P}$, $\{\tilde{c}_i\}_{j=1}^{M}$, $\{\tilde{d}_i\}_{j=1}^{M}$)\;\\
    For $x\in [a,b]$ return $\tilde{F}(x)$ as in $\ref{subsub1}$
    \caption{The algorithm returns a data structure $DS$ that represents a set valued function interpolating a set of samples at $X$.}
    \label{alg:Create_Set_Valued_Interpolant_Data_Structure}
\end{algorithm}

\begin{thebibliography}{10}
\bibitem{AbStegun}
Abramowitz, Milton, and Irene A. Stegun. "Handbook of Mathematical Functions with Formulas, Graphs, and Mathematical Tables. National Bureau of Standards Applied Mathematics Series 55. Tenth Printing." (1972).
\bibitem{Bajaj}
Bajaj, Chandrajit L., Edward J. Coyle, and Kwun-Nan Lin. "Arbitrary topology shape reconstruction from planar cross sections." Graphical models and image processing 58.6 (1996): 524-543.
\bibitem{Boissonnant}
Boissonnat, Jean-Daniel. "Shape reconstruction from planar cross sections." Computer vision, graphics, and image processing 44.1 (1988): 1-29.
\bibitem{practical_guide_to_splines}
C. R. de Boor, A Practical Guide to Splines Revised Edition, Springer, (2001).
\bibitem{Dingle} 
Dingle, Robert B. Asymptotic expansions: their derivation and interpretation. Academic Press, 1973.
\bibitem{approximations_of_set_valued_functions}
N. Dyn, E. Farkhi, and A. Mokhov,
Approximations of set-valued functions by metric
linear operators, Constr. Approx 25, (2007) 193-209.
\bibitem{approximation_of_set_valued_functions} 
N. Dyn, E. Farkhi, and A. Mokhov, A., Approximation of Set-Valued Functions: Adaptation of Classical Approximation Operators, Imperial College Press (2014).
\bibitem{numerical_analysis}
W. Gautschi, Numerical Analysis: Second Edition, Birkhäuser, (2012).
\bibitem{lebesgue_constants}
L. Brutman, On the Lebesgue function for polynomial interpolation, SIAM J. Numer. Anal. 15, (1978) 694–704.
\bibitem{uber_empirische_funktionen_und_die_interpolation}
 C. Runge, Über empirische Funktionen und die Interpolation zwischen äquidistanten Ordinaten. 46, (1901) 224–243.
\bibitem{KelsDyn}
Kels, Shay, and Nira Dyn. "Reconstruction of 3D objects from 2D cross-sections with the 4-point subdivision scheme adapted to sets." Computers and Graphics 35.3 (2011): 741-746.
\bibitem{Levin1986}
D. Levin, Multidimensional reconstruction by set-valued approximations. IMA Journal of Numerical Analysis 6, no. 2, (1986) 173-184.
\bibitem{derivative_error_bounds_for_lagrange_interpolation}
Gary W. Howell. Derivative error bounds for Lagrange interpolation: an extension of Cauchy’s bound for the error of Lagrange interpolation. Journal of Approximation Theory 67, (1991) 164-173. 
\bibitem{approximating_piecewise_smooth_functions}
Y. Lipman and D. Levin, Approximating piecewise-smooth functions. IMA Journal of Numerical Analysis, Volume 30, Issue 4, October (2010) 1159–1183.
\bibitem{BL} L. Brutman, Lebesgue functions for polynomial interpolation—a survey, Ann. Numer. Math. 4 (1997) 111–127.


\end{thebibliography}
\end{document}